\documentclass[12pt]{article}

\usepackage{
   amsthm,
   amsmath,
   amssymb,
   graphicx,
   hyperref,
  mathtools,
  fullpage
   }
\usepackage[shortlabels]{enumitem}

\usepackage{mathrsfs}
\usepackage[margin=0pt]{subcaption}
\usepackage{stmaryrd}
\usepackage{tikz-cd}
\usepackage{tikz}

\theoremstyle{plain}
\newtheorem{thm}{Theorem}[section]
\newtheorem{lemma}[thm]{Lemma}
\newtheorem{propn}[thm]{Proposition}
\newtheorem{cor}[thm]{Corollary}
\theoremstyle{definition}
\newtheorem{defn}[thm]{Definition}
\newtheorem{rmk}[thm]{Remark}
\newtheorem{ex}[thm]{Example}

\newcommand{\R}{\mathbb{R}}
\newcommand{\Z}{\mathbb{Z}}
\newcommand{\C}{\mathbb{C}}
\newcommand{\N}{\mathbb{N}}
\newcommand{\dom}{\operatorname{dom}}

\newcommand{\id}{\mathrm{id}}
\newcommand{\mat}[1]{\left( \begin{smallmatrix} #1 \end{smallmatrix} \right)}
\newcommand{\supp}{\operatorname{supp}}

\numberwithin{equation}{section}

\begin{document}

\title{H-Unitality of Smooth Groupoid Algebras}
\author{Michael D. Francis}

\maketitle

\begin{abstract}
We show that the convolution algebra of smooth, compactly-supported functions on a Lie groupoid is H-unital in the sense of Wodzicki.  We also prove H-unitality of infinite order vanishing  ideals associated to invariant, closed subsets of the unit space. This furthermore gives H-unitality for the quotients by such ideals, which are noncommutative algebras of Whitney functions. These results lead immediately to excision properties in discrete Hochschild and cyclic homology around invariant, closed subsets. This work extends previous work of the author establishing the Dixmier-Malliavin theorem in this setting.
\end{abstract}

\section{Introduction}

Let $A$   denote an associative algebra over $\C$. We do not assume $A$ is commutative or unital. Let us say  that $A$ has the \emph{weak factorization property} if every $a \in A$ can be expressed as a finite sum $a = \sum b_i c_i$, where $b_i,c_i \in A$. Notice that every unital algebra has the weak factorization property, so this notion is only of interest in the nonunital setting.

Recall that, given a Lie group $G$ equipped with Haar measure, the space $C_c^\infty(G)$ of smooth, compactly-supported functions on $G$ becomes an algebra with respect to convolution. This algebra is nonunital unless $G$ is discrete (the unit wants to be the Dirac mass at $1$). In a 1978 paper, Dixmier-Malliavin proved the following striking result.

\begin{thm}[3.1 Th\'{e}or\`{e}me, \cite{Dixmier-Malliavin}]\label{DMthm}
For any Lie group $G$, the smooth convolution algebra  $C_c^\infty(G)$ has the weak factorization property.
\end{thm}

The main technical ingredient of their proof  is the following lemma whose own proof is an intricate piece of hard  analysis.

\begin{lemma}[2.5 Lemme, \cite{Dixmier-Malliavin}]\label{DMlemma}
For any sequence $(c_m)_{m=0}^\infty$ of positive scalars, there exist $f_0,f_1 \in C_c^\infty(\R)$ and scalars $(a_m)$ with $|a_m| \leq c_m$ such that $\delta_0 = f_0 +  \sum_{m=0}^\infty a_m f_1^{(m)}$. Here,  $\delta_0$ denotes the Dirac mass at $0$ and the series converges in the sense  of compactly-supported distributions. The functions $f_0, f_1$ may be chosen to be supported in any fixed neighbourhood of the origin.
\end{lemma}
The (already nontrivial) $G=\R$ case of Theorem~\ref{DMthm} follows quite directly from Lemma~\ref{DMlemma}. Given any $\varphi \in C_c^\infty(\R)$, by choosing the coefficients $a_m$  to converge sufficiently quickly to $0$, one has that $\varphi_1 \coloneqq \sum_m a_m \varphi^{(k)}$ belongs to $C_c^\infty(\R)$ and so
\[ \varphi  = \delta_0 * \varphi =  f_0 * \varphi + \sum_{m=0}^\infty a_m f_1^{(m)} * \varphi =  f_0 * \varphi + f_1 *  \varphi_1. \]
To prove Theorem~\ref{DMthm} for a general Lie group, one writes $G$ (locally) as a product of 1-parameter subgroups and factors one group at a time.

In \cite{Francis[DM]}, the author extended Dixmier-Malliavin's result to the setting of Lie groupoids:

\begin{thm}[Theorem~5.2, \cite{Francis[DM]}]\label{DMprev}
For any Lie groupoid $G$, the smooth convolution algebra $C_c^\infty(G)$  has the weak factorization property.
\end{thm}

We always assume our Lie groupoids $G$ are equipped with a smooth Haar system so that a convolution product on  $C_c^\infty(G)$ is defined (this arbitrary choice can be avoided by working with appropriate densities in place of functions, see \cite{Connes[BOOK]},~Section~2.5). The total space of $G$ (but not its unit space) is permitted to be non-Hausdorff. As usual, in the non-Hausdorff case, $C_c^\infty(G)$ is defined somewhat differently as the span of all $C_c^\infty(U)$, where $U$ is a chart neighbourhood in $G$ (see, e.g. \cite{Connes[BOOK]}, \cite{Connes[1978]}, \cite{Khoshkam-Skandalis[2004]}).

In the case of groupoids, there is an additional phenomenon (absent in the group case) of ideals associated to invariant subsets of the unit space. If $Z$ is a closed, invariant subset of the unit space of $G$ and $G_Z \subseteq G$ denotes the closed subgroupoid consisting of arrows  whose source and target lie in $Z$, we denote by  $J_Z^\infty \subseteq C_c^\infty(G)$ the ideal (with respect to convolution) of functions which vanish to infinite order on $G_Z$. Weak factorization was also established by the author for these ideals, in the case where $Z$ is a closed invariant submanifold:

\begin{thm}[Theorem~7.1, \cite{Francis[DM]}]\label{idealDM}
For any Lie groupoid $G$, for any closed, invariant submanifold $Z$ of the unit space of $G$, the corresponding ideal $J_Z^\infty \subseteq  C_c^\infty(G)$ has the weak factorization property. 
\end{thm}

The purpose of this article is to strengthen Theorem~\ref{DMprev} and Theorem~\ref{idealDM} by showing that $C_c^\infty(G)$ and $J^\infty_Z$ are in fact  homologically unital in the sense defined by Wodzicki.  This immediately implies excision results for the Hochschild and cyclic homology of these algebras.   Applying standard permanence properties, we obtain H-unitality and excision results for noncommutative algebras of Whitney functions as well.

The present article is also more general than \cite{Francis[DM]} in that $Z$ is permitted to  be an arbitrary  invariant, closed subset of the unit space, and not necessarily a submanifold. Furthermore, an increased amount of care is taken to ensure that the results and arguments apply in the case of non-Hausdorff Lie groupoids; the results in \cite{Francis[DM]} do apply in the non-Hausdorff setting, but the verification of this is mostly left to the reader.

We give a brief review of the homological notions at play. For simplicity, we only consider complex scalars,  $\C$ being the ground field for all algebras of interest here. In general, this restriction is not necessary, though working with a ground field of characteristic zero does simplify definitions somewhat.

Given an associative algebra $A$ over $\C$, the \emph{bar complex} of $A$ is the chain complex
\[ 
\begin{tikzcd}
\ldots \ar[r,"{d'}"]  & A^{\otimes 3} \ar[r,"{d'}"] & A^{\otimes 2} \ar[r,"{d'}"] & A \ar[r] & 0, 
\end{tikzcd}\]
where the differential $d'$ is determined by
\begin{align*}
d'(a_0\otimes\ldots\otimes a_n) &= \sum_{i=0}^{n-1}  (-1)^i a_0 \otimes \ldots \otimes a_ia_{i+1} \otimes \ldots \otimes a_n &&  n \geq 1.
\end{align*}

\begin{rmk}
If one modifies the above definition of $d'$  to include the ``wrap around term'' $a_n a_0 \otimes \ldots \otimes a_{n-1}$, then one gets the differential  $d$ involved in the definition of Hochschild and cyclic homology. 
\end{rmk}

The following terminology was introduced in the groundbreaking work of  Wodzicki  on excision in  cyclic homology and algebraic K-theory (\cite{Wodzicki2}, \cite{Wodzicki}).

\begin{defn}
An associative algebra $A$ over $\C$ is said to be \emph{homologically unital} or \emph{H-unital} if its bar complex  is acyclic. That is, if
\[ 
\begin{tikzcd}
\ldots \ar[r,"{d'}"]  & A^{\otimes 3} \ar[r,"{d'}"] & A^{\otimes 2} \ar[r,"{d'}"] & A \ar[r] & 0 
\end{tikzcd}\]
is an exact sequence.
\end{defn}

As for permanence properties, one has that H-unital algebras over $\C$ are closed under under taking quotients and extensions. In general, this is only true for so-called ``pure extensions'', a hypothesis that is automatic when working over $\C$.

\begin{thm}[Corollary~3.4, \cite{Wodzicki}]\label{permanence}
Let  $0 \to A \to B \to C \to 0$ be a short exact sequence of $\C$-algebras, where  $A$  is H-unital. Then, $B$ is H-unital if and only if $C$ is H-unital.
\end{thm}

\begin{ex}
Every unital algebra $A$ is H-unital. Indeed, $x \mapsto 1 \otimes x : A^{\otimes n} \to A^{\otimes (n+1)}$ defines a contracting chain homotopy for the bar complex. More generally, if $A$ is \emph{locally unital} in the sense that, for any finite set of elements, there is an element acting as a unit for those elements, then $A$ is  H-unital. 
\end{ex}
\begin{ex}Any Banach algebra with a bounded approximate unit is H-unital. In particular, all C*-algebras are H-unital. 
\end{ex}

\begin{ex}
More in the spirit of the present article, if $M$ is a smooth manifold and $Z \subseteq M$ is any closed subset, then the ideal $I_Z^\infty \subseteq C^\infty(M)$ of smooth functions which vanish with all their derivatives on $Z$ is H-unital. Invoking Theorem~\ref{permanence}, one has that the algebra $\mathcal{E}^\infty(Z)\coloneqq C^\infty(M)/I_Z^\infty$ of Whitney functions on $Z$ is H-unital. Theorem~\ref{permanence} also gives results in the relative setting; if $Z \subseteq Y$ are closed subsets of $M$, then the kernel $I^\infty_{Z, Y} \cong I^\infty_Z/I^\infty_Y$ of the restriction map $\mathcal{E}^\infty(Y) \to \mathcal{E}^\infty(Z)$ is H-unital. 
\end{ex}

The importance of H-unitality stems from its close relationship to the excision problem for Hochschild and cyclic homology.

\begin{thm}[3.1~Theorem, \cite{Wodzicki}]\label{wod}
Let $A$ be an associative algebra over $\C$. Then, the following are equivalent:
\begin{enumerate}
\item $A$ has the excision property for Hochschild homology.
\item $A$ has the excision property for cyclic homology.
\item $A$ is H-unital.
\end{enumerate}
\end{thm}

For a precise explanation of the terminology above, one may see  \cite{Wodzicki}, or Section~1.4 of \cite{Loday}. The key point  is the following: if $I$ is an H-unital algebra, and $B$ contains $I$ as an ideal, then  the short exact sequence of $\C$-algebras
\[ \begin{tikzcd} 
0 \ar[r] & I  \ar[r] & B \ar[r] & B/I \ar[r] & 0
\end{tikzcd} \]
induces  a corresponding long exact sequence in Hochschild/cyclic homology (since $\C$ is a field of characteristic zero, the purity hypothesis in \cite{Wodzicki} is automatically satisfied).

\begin{rmk}
In the present article, all homology theories being considered are  their ``discrete versions'' and  depend only on the underlying algebra. In the discrete case,  the connection between excision and H-unitality is neatly explained in  \cite{Wodzicki}.  On the other hand, the continuous setting has a somewhat more complicated history. The  planned sequel to Wodzicki's paper addressing continuous counterparts to Theorem~\ref{wod} did not appear  (see Remark~8.5~(2)), so that the most general results have sometimes remained folklore, and primarily known to experts. Indeed, the theory has not always been adequate for the most interesting applications. See, for instance,  \cite{Brasselet-Pflaum} and Remark~6.5 therein. Several authors have worked to fill in the gaps in literature (see \cite{Guccione} and \cite{Meyer}) with the consequence that the connection between H-unitality and excision in the continuous setting is now on much firmer footing.
\end{rmk}

Note that, if $A$ is H-unital, then, in particular, the mapping $A\otimes A \to A$ is surjective. That is, H-unital algebras have the weak factorization property. In this article, we prove the following result which may be viewed as a generalization of the Dixmier-Malliavin theorem for Lie groupoids of \cite{Francis[DM]}.

\begin{thm}\label{mainthm1}
For any Lie groupoid $G$, the smooth convolution algebra $C_c^\infty(G)$ is H-unital.
\end{thm}

\begin{rmk}
The author would like to thank Xiang Tang for pointing out the similarity of Theorem~\ref{mainthm1} above to Proposition~2 of \cite{Crainic-Moerdijk}. We note, however, that, because the tensor products in this article are algebraic tensor products and the tensor products in \cite{Crainic-Moerdijk} are completed tensor products $\overline\otimes$ satisfying $C_c^\infty(M) \overline\otimes C_c^\infty(N) = C_c^\infty(M \times N)$  for all  smooth manifolds $M$ and $N$, the results are   distinct.
\end{rmk}

The second purpose of this article is to prove the following generalization of Theorem~\ref{idealDM}.

\begin{thm}\label{mainthm2}
For any Lie groupoid $G$, for any closed, invariant subset $Z$ of the unit space of $G$,  the corresponding ideal $J_Z^\infty \subseteq C_c^\infty(G)$ is H-unital.
\end{thm}

The main practical consequence of Theorem~\ref{mainthm2} is that H-unitality of $J_Z^\infty$ combined with Theorem~\ref{wod} yields the following important corollary.

\begin{cor}\label{excision1}
For any Lie groupoid $G$, for any closed, invariant subset $Z$ of the unit space of $G$, the exact sequence
\[ \begin{tikzcd}
0 \ar[r] & J_Z^\infty \ar[r] & C_c^\infty(G) \ar[r] & C_c^\infty(G)/J_Z^\infty \ar[r] & 0
\end{tikzcd} \]
induces corresponding long exact sequences in Hochschild and cyclic homology.
\end{cor}

It is hoped that the above will lead to an improved  understanding of localization around invariant subsets in  calculations of the cyclic and Hochschild homology of convolution algebras of Lie groupoids. Examples of calculations utilizing this excision principle will be considered elsewhere.  Such calculations fall squarely within Connes' noncommutative geometry program. One may see \cite{Pflaum-Posthuma-Tang} for recent progress in this area.

Applying Theorem~\ref{permanence},  the quotient
\[ \mathcal{E}_c^\infty(G_Z) \coloneqq C_c^\infty(G)/J_Z^\infty \]
is also H-unital. This  is exactly  the algebra of compactly-supported Whitney functions on the closed subgroupoid $G_Z \subseteq G$ of arrows with source and target in $Z$ (the notation $ \mathcal{E}_c^\infty(G_Z)$ suppresses the dependence on the embedding of $G_Z$ in $G$).  It should also be emphasized that, here, the product on $\mathcal{E}_c^\infty(G_Z)$ is the noncommutative one descending from convolution, rather than the usual (commutative)  product of Whitney functions. 

Given a nested pair $Z \subseteq Y$   of closed, invariant subsets of the unit space of $G$, one may also consider the ideal
\[ J_{Z,Y}^\infty \coloneqq \ker(\mathcal{E}_c^\infty(G_Y) \to \mathcal{E}_c^\infty(G_Z)). \] 
Because $J_{Z,Y}^\infty \cong J_Z^\infty/J_Y^\infty$, Theorem~\ref{mainthm2}  and Theorem~\ref{permanence} imply that  $J_{Z,Y}^\infty$ is H-unital as well. Thus, using Theorem~\ref{wod}, we obtain following relative analog of Corollary~\ref{excision1}.
\begin{cor}\label{excisionrel}
For any Lie groupoid $G$, for any nested pair $Z\subseteq Y$ of closed, invariant subsets  of the unit space of $G$,  the short exact sequence
\[ \begin{tikzcd}
0 \ar[r] & J_{Z,Y}^\infty \ar[r] & \mathcal{E}_c^\infty(G_Y) \ar[r] & \mathcal{E}_c^\infty(G_Z) \ar[r] & 0
\end{tikzcd} \]
induces corresponding long exact sequences in Hochschild and cyclic homology.
\end{cor}

From a certain point of view, Corollary~\ref{excisionrel} is more compelling than  Corollary~\ref{excision1} because algebras of Whitney functions are  somewhat more algebraic objects, making consideration of their discrete homology more natural.

We now discuss the criterion that will be used to establish H-unitality. As mentioned above, if an algebra $A$ is unital, then $x \mapsto 1 \otimes x : A^{\otimes n}\to A^{\otimes (n+1)}$ defines a contracting homotopy for the bar complex. More generally, if there exists a $\C$-linear map $\phi : A \to A \otimes A$ which is right $A$-linear (where the  right $A$-module structure on $A\otimes A$ is such that  $(a \otimes b)\cdot c = a\otimes(bc)$ for  $a,b,c \in A$) and makes the diagram
\[ \begin{tikzcd}
A \ar[rd,equals] \ar[r,"{\phi}"] & A \otimes A \ar[d,"\text{multiplication}"] \\
 & A
\end{tikzcd} \]
commutative, then a simple calculation shows that $\phi \otimes \id  : A^{\otimes n}\to A^{\otimes (n+1)}$ gives a contracting homotopy for the bar complex. In fact, because any cycle only involves finitely any elements of $A$, it is not actually necessary to have a globally-defined map $\phi$; one can get by with a family of locally-defined maps. This is formalized in the  following result 
from  \cite{Wodzicki} (specialized to the case $k=\C$) and we refer the reader to that source for the proof.

\begin{propn}[4.1 Proposition, \cite{Wodzicki}]\label{suffcond}
Let $A$ be an associative $\C$-algebra. Suppose that, for every finite set $\mathscr{P} \subseteq A$, there exists a right ideal $A_0 \subseteq A$ such that $\mathscr{P} \subseteq A_0$ and a $\C$-linear map  $\phi : A_0 \to A \otimes A$ such that:
\begin{enumerate}
\item $\phi$ is a map of right $A$-modules (the right $A$-module structure on $A\otimes A$ is such that $(a \otimes b)\cdot c = a\otimes(bc)$ for $a,b,c \in A$).
\item The diagram \[ \begin{tikzcd}
A_0 \ar[rd,"\text{inclusion}"'] \ar[r,"{\phi}"] & A \otimes A \ar[d,"\text{multiplication}"] \\
 & A
\end{tikzcd} \]
is commutative.
\end{enumerate}
\end{propn}

One straightforward consequence of Proposition~\ref{suffcond} is H-unitality of  ``locally unital'' algebras;  if   every finite subset of $A$ admits a common left unit, then $A$ is H-unital. While the latter statement suffices for many  algebras of interest   in noncommutative geometry, the algebras considered here are not locally unital.  Already, the convolution algebra $C_c^\infty(\R)$ is not locally unital (under Fourier transform, it is an algebra of analytic functions under pointwise multiplication). Nonetheless, our proofs of H-unitality will  pass through the sufficient condition of Proposition~\ref{suffcond}. As an instructive example, we establish H-unitality for $C_c^\infty(\R)$ below.

\begin{ex}
Let $X=\frac{d}{dx}$, considered as an operator on $C_c^\infty(\R)$. 
Given a formal series $P(z) = \sum_{m\geq 0 } a_m z^m \in \C[[z]]$, one may define $P(X) : \dom(P(X)) \to C_c^\infty(\R)$ by 
\begin{align*}
\dom(P(X)) &= \{ \varphi \in C_c^\infty(\R) : {\textstyle\sum_{m \geq 0}} |a_m| \left\| X^{m+r} \varphi   \right\| < \infty \text{ for all } r \geq 0 \} \\
P(X) \varphi &= \sum_{m \geq 0 } a_m X^m\varphi,
\end{align*}
where $\|\cdot\|$ denotes the uniform norm. One may check that $\dom(P(X))$ is an ideal in $C_c^\infty(\R)$.   Given any finite subset $\mathscr{P} \subseteq C_c^\infty(\R)$, one may use Lemma~\ref{DMlemma} to find a representation $\delta_0 = f_0 + \sum_{m=0}^\infty a_m f_1^{(m)}$ of the Dirac measure  in terms of $f_0,f_1 \in C_c^\infty(\R)$ such that $P(z) = \sum_{m \geq 0} a_m z^m$ has $\mathscr{P} \subseteq \dom(P(X))$. Then, the map  
\[ \varphi \mapsto  f_0 \otimes \varphi + f_1 \otimes P(X) \varphi : \dom(P(X)) \to C_c^\infty(\R) \otimes C_c^\infty(\R) \]
satisfies the conditions of Proposition~\ref{suffcond} by the same calculation given below  the statement of Lemma~\ref{DMlemma}, so  $C_c^\infty(\R)$ is H-unital by Proposition~\ref{suffcond}.
\end{ex}

We give a brief summary of the contents of this article. Section~2 reviews  certain aspects of the theory of non-Hausdorff smooth manifolds. A number of examples are included with the intention being to clarify where complications can arise, especially as pertains to supports of functions and flows of vector fields. Section~3 is a technical section devoted to domains of operators obtained by taking power series in a fixed set of  vector fields. The goal here is to confirm that, by taking series coefficients to be sufficiently small,  domains of such operators can be made to contain any finite set of bump functions. In Section~4, we establish conventions for Lie groupoids and smooth actions of Lie groupoids. Section~5 is concerned with our first main result, the H-unitality of the smooth convolution algebras of groupoids. Section~6 is concerned with factorization  with respect to pointwise multiplication  of smooth functions vanishing to infinite order on a closed set, relative to a submersion. Section~7 contains our second main result, the H-unitality of infinite order vanishing ideals in smooth convolution algebras of groupoids.

\section{Non-Hausdorff smooth manifolds}

It is frequently necessary  in noncommutative geometry to work with non-Hausdorff manifolds. In this section, we highlight some of the complications that can arise and review the standard workarounds. Readers who are only interested in  Hausdorff Lie groupoids may safely ignore this section. Readers who are already  well-versed in the techniques needed to deal with non-Hausdorff Lie groupoids may also prefer to skip this material.

In this article, a \textbf{smooth manifold} refers to a second-countable topological space $M$ equipped with a  smooth atlas (more precisely, an equivalence class of smooth atlases or, alternatively, a maximal smooth atlas) of some fixed dimension.  The topology of $M$ is not required to be Hausdorff. Of course, $M$ is locally Hausdorff, being locally Euclidean.

\begin{rmk}Actually, it would do  little harm to drop the assumption of second-countablity as well, provided we at least assume that every \emph{component} of $M$ is second-countable. This amounts to working with possibly uncountable disjoint unions of second countable manifolds.
\end{rmk}

Much  of the basic theory of smooth manifolds goes  through unchanged without the Hausdorff hypothesis. For example,  one can talk about smooth vector fields on and smooth maps between non-Hausdorff smooth manifolds by requiring smoothness in every chart. Many of the  issues that do arise are due to the nonexistence of partitions of unity. Ultimately, this stems from an undersupply of scalar-valued functions that look smooth in \emph{every} chart. To some extent, it is possible to define one's way around this issue by modifying the definition of bump functions. The simple definition below is due to Connes. More sophisticated approaches with better functorial properties are also possible. See the unpublished manuscript \cite{Crainc-Moerdijk}.

\begin{defn}\label{nonHausbump}
Let $M$ be a possibly non-Hausdorff smooth manifold. Then, $C_c^\infty(M)$ denotes the linear span of all functions $\varphi : M \to \C$ that are given as the extension by zero of a function of the form $f \circ \chi : U \to \C$ where $U \subseteq M$ is open, $\chi : U \to \R^d$ is a diffeomorphism, and $f \in C_c^\infty(\R^d)$.
\end{defn}

If $M$ is Hausdorff, the usual meaning of $C_c^\infty(M)$ is recovered. On the other hand, if $M$ is non-Hausdorff, the  notation is $C_c^\infty(M)$ is misleading in several ways:
\begin{enumerate}[(i)]
\item A function $\varphi \in C_c^\infty(M)$ need not be  smooth in every chart (or even continuous). \item The support of $\varphi \in C_c^\infty(M)$, defined in the usual way as the complement of the largest open set where $\varphi$ is zero (equivalently, the closure of the nonvanishing locus of $\varphi$), need not be compact. 
\item $C_c^\infty(M)$ need not be closed under the operation of  pointwise product.
\end{enumerate}

\begin{ex}\label{nonHausex}
Let $M=\Big((-\infty,0) \times \{0\}\Big) \cup \bigcup_{n \in \Z} \Big([0,\infty) \times \{n\}\Big)$, 
with smooth manifold structure determined by the atlas $\{\chi_n :U_n \to \R\}_{n \in \Z}$ defined by
\begin{align*}
U_n  = \left((-\infty,0) \times \{0\}\right) \cup \left([0,\infty) \times \{n\} \right) &&
\chi_n = \mathrm{proj}_1|_{U_n} && n \in \Z.
\end{align*}
Choose some $f \in C_c^\infty(\R)$ with $\mathrm{supp}(f)=[-1,1]$ and define $\varphi_n$ to be the extension by zero to all of $M$ of $f \circ \chi_n$, so that $\varphi_n \in C_c^\infty(M)$ by definition. Then,
\begin{enumerate}[(i)]
\item The restriction of $\varphi_m$ to $U_n$ is not smooth if $m \neq n$.
\item The support  of $\varphi_0$ is $\left([-1,1] \times \{0 \} \right) \cup \{ (0,n) : n \in \Z \}$, which is not compact.
\item The pointwise product of $\varphi_m$ and $\varphi_n$ does not belong to $C_c^\infty(M)$ if $m \neq n$.
\end{enumerate}
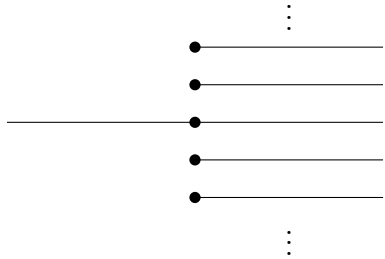
\begin{center}
\begin{tikzpicture}[scale=.5]
\draw (-5,0)-- (5,0);
\draw (0,1)-- (5,1);
\draw (0,2)-- (5,2);
\draw (0,-1)-- (5,-1);
\draw (0,-2)-- (5,-2);
\node at (0,0) [circle,fill,inner sep=1.5pt]{};
\node at (0,1) [circle,fill,inner sep=1.5pt]{};
\node at (0,2) [circle,fill,inner sep=1.5pt]{};
\node at (0,-1) [circle,fill,inner sep=1.5pt]{};
\node at (0,-2) [circle,fill,inner sep=1.5pt]{};
\node at (2.5,3) {$\vdots$};
\node at (2.5,-3) {$\vdots$};
\end{tikzpicture}
\end{center}
\end{ex}

Note as well that the space in the example above does arise naturally in geometry. For example, foliate the lower half of the cylinder $S^1 \times \R$ into circles and the upper half into spirals so that there is one-sided holonomy at the equator $S^1 \times \{0\}$. The restriction of the holonomy groupoid of this foliation to a vertical line  transversal is  diffeomorphic to $M$.

Although it can occur that $\supp(\varphi) \coloneqq \overline{ \varphi^{-1}(\C^\times)}$ is noncompact for $\varphi \in C_c^\infty(M)$ when $M$ is noncompact, we do at least have the following.

\begin{propn}\label{pointset}
Let $\tau:M \to B$ be a smooth map of smooth manifolds, where $B$ is Hausdorff and $M$ is possibly non-Hausdorff. Then, for all $\varphi \in C_c^\infty(M)$, we have that $\tau(\supp(\varphi))$ is compact.
\end{propn}
\begin{proof}
Write $\varphi = \sum_{i=1}^n (f_i \circ \chi_i)_0$ where $\chi_i : U_i \to \R^d$ are charts, $f_i \in C_c^\infty(\R^d)$, and the $0$ subscripts denote the operation of extension by zero. Then, $K = \bigcup_{i=1}^n \chi_i^{-1}(\supp(f_i))$ is a compact (but possibly not closed) subset of $M$ containing $\varphi^{-1}(\C^\times)$. The conclusion follows from the elementary point-set topological lemma below, with $S = \varphi^{-1}(\C^\times)$.
\end{proof}

\begin{lemma}
Let $f : X \to Y$ be a continuous map of topological spaces, where $Y$ is Hausdorff and $X$ is possibly non-Hausdorff. Suppose that $S \subseteq K \subseteq X$ and $K$ is compact (but possibly not closed). Then, $f(\overline S) = \overline{ f(S)}$ and this is a compact subset of $Y$. 
\end{lemma}
\begin{proof}
We have $f\left(\overline S\right) \subseteq \overline{f(S)} \subseteq \overline{f\left(\overline S \cap K\right)} = f\left(\overline S \cap K\right) \subseteq f\left(\overline S\right)$, so all these sets are equal to the compact set $f(\overline S \cap K)$.
\end{proof}

Another issue with Definition~\ref{nonHausbump} is the poor control that the support of an element $\varphi \in C_c^\infty(M)$ exerts over the supports of the summands in the possible decompositions  $\varphi$ into functions coming from charts. One might expect that, if $U$ is open and $\mathrm{supp}(\varphi) \subseteq U$, then $\varphi \in C_c^\infty(U) \subseteq C_c^\infty(M)$, i.e. $\varphi$ can be expressed in terms of bump functions defined on chart neighbourhoods contained in $U$. In general, this is not assured, as the example below illustrates.

\begin{ex}\label{supbad}
Fix a smooth function $\theta : \R \to \R$ with $\theta(x)=0$ for $x \leq 0$ that restricts to a  diffeomorphism $\theta_+:(0,\infty)\to(0,\infty)$. Consider the non-Hausdorff smooth manifold $M = \R \cup_{\theta_+} \R$ obtained by using $\theta_+$ to glue two copies of $\R$ along $(0,\infty)$.  More precisely, 
\begin{align*}
M&=\left((-\infty,0] \times \{-1,1\}  \right)\cup \left( (0,\infty) \times \{0\} \right)  \\
U_+ &= \left((-\infty,0] \times \{1\}  \right)\cup \left( (0,\infty) \times \{0\} \right)\\
U_- &= \left((-\infty,0] \times \{-1\}  \right)\cup \left( (0,\infty) \times \{0\} \right) 
\end{align*}
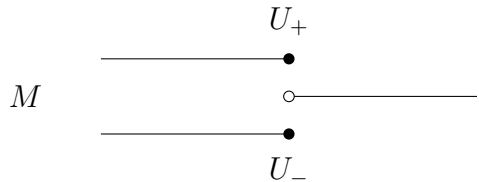
\begin{center}
\begin{tikzpicture}[scale=.5]
\node at (-7,0) {$M$};
\node at (0,2) {$U_+$};
\node at (0,-2) {$U_-$};
\draw (-5,1)-- (0,1);
\draw (-5,-1)-- (0,-1);
\draw (.14,0)-- (5,0);
\node at (0,1) [circle,fill,inner sep=1.5pt]{};
\node at (0,-1) [circle,fill,inner sep=1.5pt]{};
\draw (0,0) circle (4pt);
\end{tikzpicture}
\hfill
\end{center}
with non-Hausdorff smooth manifold structure determined by the two charts:
\begin{align*}
\chi_+: U_+ \to \R && \chi_+(x,y)&= x  \\
\chi_-: U_- \to \R && \chi_-(x,y) &= \begin{cases}
x & x \leq 0 \\
\theta_+(x) & x > 0.
\end{cases} 
\end{align*}
Fix  $g \in C_c^\infty(\R)$ with $\mathrm{supp}(g)=[-1,1]$ such that $g(x)=x$ for all $x$ in a neighbourhood of $0$. Define $f \in C_c^\infty(\R)$ by $f(x)=g(\theta(x))$ for $x>0$ and $f(x)=0$ for $x \leq 0$. Let $\varphi \in C_c^\infty(M)$ be given by $\varphi = (g \circ \chi_-)_0 - (f \circ \chi_+)_0$, where the subscript $0$ denotes extension by zero.
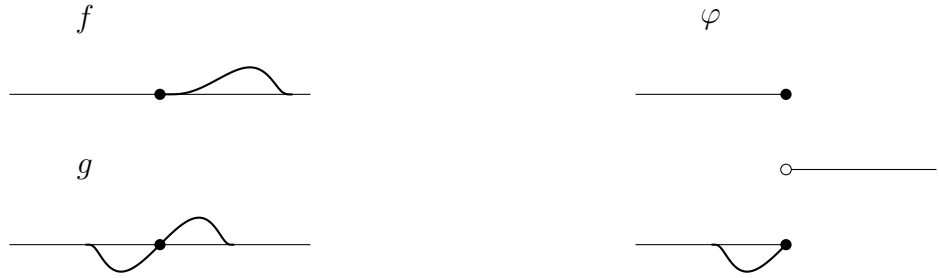
\begin{center}
\hfill
\begin{tikzpicture}[scale=1]
\draw (-2,1)-- (2,1);
\draw (-2,-1)-- (2,-1);
\node at (0,1) [circle,fill,inner sep=1.5pt]{};
\node at (0,-1) [circle,fill,inner sep=1.5pt]{};
\node at (-1,2) {$f$};
\node at (-1,0) {$g$};
\draw[scale=1, black, thick, domain=.01:1.762, smooth, variable=\x] plot ({\x}, {1+\x*exp(-1/\x)*exp(1-1/(1-\x*exp(-1/\x)*\x*exp(-1/\x)))});
\draw[scale=1, black, thick, samples=101, domain=-.99:.99, smooth, variable=\x] plot ({\x}, {\x*exp(1-1/(1-\x*\x))-1});
\end{tikzpicture}
\hfill
\begin{tikzpicture}[scale=1]
\node at (-1,2) {$\varphi$};
\draw (-2,1)-- (0,1);
\draw (-2,-1)-- (0,-1);
\draw (0.07,0)-- (2,0);
\node at (0,1) [circle,fill,inner sep=1.5pt]{};
\node at (0,-1) [circle,fill,inner sep=1.5pt]{};
\draw (0,0) circle (2pt);
\draw[scale=1,  black, thick, samples=101, domain=-.99:0, smooth, variable=\x] plot ({\x}, {\x*exp(1-1/(1-\x*\x))-1});
\end{tikzpicture}
\hfill
\end{center}
Then, $\mathrm{supp}(\varphi) = [-1,0]  \times \{-1\}$  (note the point $(0,1)$ does not belong to the support because $\varphi$ vanishes identically on $U_+$). In particular, $\mathrm{supp}(\varphi) \subseteq U_-$. However, $\varphi$ does not belong to $C_c^\infty(U_-) \subseteq C_c^\infty(M)$.
\end{ex}

Another issue with Definition~\ref{nonHausbump} of relevance to us relates to smooth vector fields. On the one hand, there is no difficulty defining the tangent bundle $TM$ of a non-Hausdorff smooth manifold $M$ and defining a vector field $X$ to be a section of $TM$  which is smooth in every chart. On the other hand, integral curves of $X$ may not be unique, preventing one from talking about the flow of $X$. In fact things already go wrong at infinitesimal level, i.e. one does not even have a well-defined linear map $X:C_c^\infty(M)\to C_c^\infty(M)$. The  operators defined by $X$ in charts need not patch together. This issue arises already in the simplest examples, as shown in Example~\ref{badvecfld} below.

\begin{ex}\label{badvecfld}
Let $M=\R^\times \cup \{0_1,0_2\}$, the ``line with two origins'' with noncommutative smooth manifold structure coming from the two obvious charts $\chi_1,\chi_2:M\to\R$.  Let $X$ be the (global) smooth vector field on $M$ which coincides with $\frac{d}{dx}$ in both charts. Fix $f \in C_c^\infty(\R)$ with $f(0)=0$ and $f'(0)=1$. For $i=1,2$, let $\varphi_i \in C_c^\infty(M)$ be the extension by zero of $f \circ \chi_i$ and let $\psi_i \in C_c^\infty(M)$ be the extension by zero of $f' \circ \chi_i$. Then, $\varphi_1 - \varphi_2 = 0$, but $\psi_1 - \psi_2$ is the function which is zero on $\R^\times$, $1$ at $0_1$ and $-1$ at $0_2$. Thus, $\varphi_1 - \varphi_2 = 0$ and $\psi_1 - \psi_2\neq 0$. This shows that the result of applying   $X$ to the zero function is not unambiguously defined.
\end{ex}

A \textbf{flow} for a smooth vector field $X$ on a smooth manifold $M$ is a smooth map $(t,m) \mapsto \phi_t : W \to M$ where  $W \subseteq \R \times M$ is an open set containing $\{0\} \times M$ whose intersection with $\R \times \{m\}$ is connected for all $m \in M$ such that $\phi_0(m)=m$ for all $m \in M$ and $\frac{d}{dt} \phi_t(m) = X(m)$ for all $(t,m) \in W$. It is easy to see that, if $X$ has a flow, then any two flows agree on the intersection of their domains, and there is a unique maximal flow $(t,m)\mapsto \phi_t(m) : \R \times M$ which furthermore satisfies $\phi_{s+t}(m)=\phi_s(\phi_t(m))$ whenever $(t,m) , (s,\phi_t(m)) \in W$.

\begin{defn}\label{nonbranchingdef}
Let $M$ be a smooth manifold, not necessarily Hausdorff, and let $X$ be a smooth vector field on $M$. We say that $X$ is \textbf{non-branching} if it has a flow.
\end{defn}

A non-branching vector field $X$ with flow $\phi$ does determine a well-defined linear map $X:C_c^\infty(M)\to C_c^\infty(M)$. One way to see this is to note that the locally defined operators on charts agree with the global one defined by $(Xf)(p)=\lim_{t\to0} (f(\phi_t(p)-f(p))/t$, so they patch together. Indeed, one may check that existence of the flow of $X$ is equivalent to its well-definedness as an operator on $C_c^\infty(M)$.  Note the equation $(Xf)(p)=\lim_{t\to0} (f(\phi_t(p)-f(p))/t$ can be also be used to justify the following expected  fact that $X$, viewed as an operator on $C_c^\infty(M)$, does not increase supports.

\begin{propn}\label{nosuppinc}
If $X$ is a smooth, non-branching vector field on a possibly non-Hausdorff smooth manifold $M$, then $\supp(X \psi) \subseteq \supp(\psi)$ for all $\psi \in C_c^\infty(M)$. \qed
\end{propn}

If $X$ is  a \emph{complete} non-branching vector field, we   consider its flow as a smooth, 1-parameter group of diffeomorphisms of $M$  and  denote it by $t \mapsto e^{tX}$. In other words:
\begin{align*}
\tfrac{d}{dt} e^{tX} m =  X(m) && m \in M 
\end{align*}

In spite of the  shortcomings highlighted in Examples~\ref{nonHausex}, \ref{supbad} and \ref{badvecfld}, Definition~\ref{nonHausbump} does allow  one to bypass many issues relating to nonexistence of partitions of unity. The following simple lemma will suffice for many purposes.

\begin{lemma}[Cf. \cite{Khoshkam-Skandalis[2002]}, Lemma~1.3.]\label{cover}
Let $M$ be a possibly non-Hausdorff smooth manifold. Let $(U_i)_{i \in I}$ be a cover of $M$ by Hausdorff open sets. Then, every $\varphi \in C_c^\infty(M)$ can be expressed as a finite sum $\sum \varphi_i$ where $\varphi_i$ is the extension by zero to all of $M$ of some $f_i \in C_c^\infty(U_i)$.
\end{lemma}
\begin{proof}
Without loss of generality, $\varphi$ is the extension by zero to all of $M$ of some $f \in C_c^\infty(U)$, where $U \subseteq M$ is open and Hausdorff.  Then, by usual theory of Hausdorff smooth manifolds, we can write $f$ as a finite sum $\sum f_i$ where $f_i \in C_c^\infty(U_i \cap U) \subseteq C_c^\infty(U_i)$. Then, letting $\varphi_i$ be the extension by zero to all of $M$ of $f_i$, we have $\varphi = \sum \varphi_i$.
\end{proof}

The following proposition shows some other important respects in which $C_c^\infty(M)$ is well-behaved.

\begin{propn}\label{convsetup} \text{ }
\begin{enumerate}
\item If $M$ and $N$ are possibly non-Hausdorff smooth manifolds, $f \in C_c^\infty(M)$, $g \in C_c^\infty(N)$, then $f \otimes g \in C_c^\infty(M\times N)$ (here, $(f\otimes g)(m,n)=f(m)g(n)$).
\item If $M$ is a possibly non-Hausdorff smooth manifold, and $N \subseteq M$ is a closed submanifold, then  restriction gives a surjective map $C_c^\infty(M) \to C_c^\infty(N)$.
\item If $M$ is a possibly non-Hausdorff smooth manifold and $\theta : M \to M$ is a self-diffeomorphism, then pullback along $\theta$ determines a linear bijection $C_c^\infty(M) \to C_c^\infty(M)$. 
\item If $\tau:M\to N$ is a smooth map of manifolds, with $N$ Hausdorff, then $C_c^\infty(M)$ is a $C^\infty(N)$-module with respect to $f \cdot \varphi \coloneqq (f \circ \tau)g$, $f \in C^\infty(N)$, $\varphi \in C_c^\infty(M)$.
\end{enumerate} \qed
\end{propn}
The above statements are either immediate from Definition~\ref{nonHausbump}, or follow from an application of Lemma~\ref{cover}.

\section{Infinite series of vector fields}

The following section is somewhat technical, but more or less elementary. The results obtained will be somewhat stronger than strictly  necessary, but will enable us to streamline arguments later on. Essentially we need elaborated versions of the following elementary fact: if $(f_n)_{n \geq 0}$ is a sequence in $C_c^\infty(\R)$ with uniformly bounded supports, then there exist positive scalars $(c_n)_{n \geq 0}$ such that, for any sequence of scalars $(a_n)_{n \geq 0}$ with $|a_n| \leq c_n$, the series $\sum_{n \geq 0} a_n f_n$ converges absolutely and uniformly  to a function in $C_c^\infty(\R)$. To see this, one may take, for instance, $c_n = \min_{k \leq n} \left( 2^n \|f^{(k)}_n\|\right)^{-1}$. 

The following definition is well-suited to our purposes.

\begin{defn}
Let $P(z) = \sum_{m \geq 0} a_m z^m \in \C[[z]]$ be a formal series and $X$ a smooth vector field  on $\R^d$, thought of as a differential operator $X : C_c^\infty(\R^d) \to C_c^\infty(\R^d)$. Then, we define $\dom(P(X))$ to be the set of all $\varphi \in C_c^\infty(\R^d)$ such that $\sum_{m \geq 0} |a_m| \| \tfrac{\partial^\alpha}{\partial x^\alpha} X^m \varphi \|< \infty$ for every multi-index $\alpha \in \N^d$ and put 
\[ P(X) \varphi = \sum_{m \geq 0} a_m X^m \varphi \]
for all $\varphi \in \dom(P(X))$.
\end{defn}

Some basic  consequences of this definition are collected in the following proposition.

\begin{propn}
Suppose $P(z) = \sum_{m \geq 0} a_m z^m \in \C[[z]]$ and  $X$ is a smooth vector field on $\R^d$. Define $P(X) : \dom(P(X)) \to C_c^\infty(\R^d)$ as above.
\begin{enumerate}
\item $\frac{\partial^\alpha}{\partial x^\alpha} P(X) \varphi = \sum_{m\geq0} \frac{\partial^\alpha}{\partial x^\alpha} X^m \varphi$ for all $\alpha \in \N^d$, $\varphi \in C_c^\infty(\R^d)$, where  the series converges absolutely and uniformly.
\item $\supp(P(X)\varphi) \subseteq \supp(\varphi)$ for all $\varphi \in \dom(P(X))$.
\item The definition of $P(X)$ is coordinate-independent; if $\theta : \R^d \to \R^d$ is a diffeomorphism, then $\dom(P(\theta^*(X))) = \theta^*(\dom(P(X)))$ and $P(\theta^*(X)) = \theta^*(P(X))$.
\item If $P(z) = \sum_{m \geq 0} a_m z^m, Q(z) = \sum_{m \geq 0} b_m z^m \in \C[[z]]$ and $|a_m|\leq|b_m|$ for $m\geq 0$, then $\dom(Q(X)) \subseteq \dom(P(X))$.
\end{enumerate}\qed
\end{propn}
The proof of the above proposition is straightforward and we omit it. For example, property (3) comes down to the fact that $\frac{\partial^\alpha}{\partial x^\alpha} (f \circ \theta)$ can be expressed as a finite sum $\sum_{|\beta|\leq|\alpha|} \left[ (\frac{\partial^\beta}{\partial x^\beta} f) \circ \theta \right] \cdot \theta_\beta$, where $\theta_\beta$ is a smooth function made up of partial derivatives of components of $\theta$. 

Property (3) permits us to make the following definition:

\begin{defn}
Suppose that $M$ is a not-necessarily-Hausdorff smooth manifold, $X$ is a smooth, non-branching (Definition~\ref{nonbranchingdef}) vector field on $M$ and $P(z) = \sum_{m \geq 0} a_m z^m \in \C[[z]]$.  Then, we define $\dom(P(X))$ to be the linear span of functions $\varphi$ on $M$ given as the extension by zero of functions $f \circ \chi$, where  $\chi : U \to \R^d$ is a Euclidean chart and $f \in C_c^\infty(\R^d)$ belongs to $\dom(P(Y))$, where $Y=\chi_*(X|_U)$.  We define $P(X) :  \dom(P(X)) \to C_c^\infty(M)$ by 
\[ P(X)\varphi = \sum_{m \geq 0} a_m X^m \varphi, \]
this series being absolutely and uniformly convergent for  $\varphi \in \dom(P(X)$.
\end{defn}

The lemma below will be used to control the domains of compositions of operators $P(X)$.

\begin{lemma}
Suppose that $(a_m)_{m \geq 0}$ is a sequence of nonnegative scalars. Then, there exists a non-increasing sequence of positive scalars $(c_m)_{m \geq 0}$ such that
\[ \sum_{m_1,\ldots,m_n \geq 0 } c_{m_1}\cdots c_{m_n} a_{m_1+\ldots+m_n+r} < \infty \]
for all integers $n,r$ with $n \geq 1$ and $r \geq 0$. 
\end{lemma}
\begin{proof}
Without loss of generality, the terms of $(a_m)$ are positive. We make repeated use of the fact that, given positive scalars  $(a_m^{(i)})_{m, i\geq 0}$,  there is a positive sequence $(b_m)_{m\geq 0}$ such that $\sum_{m \geq 0} b_m a_m^{(i)} < \infty$ for every $i \geq 0$. One may use, for instance, $b_m = \left( 2^m  \max_{i \leq m}  a_m^{(i)}    \right)^{-1}$. 

Applying the aforementioned fact, there is a sequence of positive scalars   $(c^{(1)}_m)_{m \geq 0}$  such that  $a_r^{(2)} \coloneqq \sum_{m \geq 0} c^{(1)}_m a_{m+r}<\infty$ for all $r \geq 0$. Applying this fact again, there is a sequence of positive scalars  $(c^{(2)}_m)_{m \geq 0}$ such that  $a_r^{(3)} \coloneqq \sum_{m \geq 0} c^{(2)}_m a_{m+r}^{(2)} = \sum_{m_1,m_2 \geq 0} c^{(1)}_{m_1} c^{(2)}_{m_2} a_{m_1+m_2+r}<\infty$ for all $r \geq 0$. Continuing in this way, we obtain positive sequences $(c_m^{(i)})_{m \geq 0}$ for $i=1,\ldots,m$ such that
\[ \sum_{m_1,\ldots,m_n \geq 0 } c_{m_1}^{(1)} \cdots c_{m_n}^{(n)} a_{m_1+\ldots+m_n+r} < \infty \]
for all $r \geq 0$.

Next, put $c_m \coloneqq \min(c_m^{(1)},\ldots, c_m^{(m)})$. We show that $(c_m)$ satisfies the conclusion of the lemma  by induction on $n \geq 1$. Since $c_m \leq c_m^{(1)}$ for all $m$, the statement holds for $n=1$. Let $n \geq 2$. Then, 
\begin{align*}
&
\sum_{m_1,\ldots,m_n \geq 0 } c_{m_1}\cdots c_{m_n} a_{m_1+\ldots+m_n+r} \\
\leq & 
\sum_{m_1,\ldots, m_n \geq n} c_{m_1}\cdots c_{m_n} a_{m_1+\ldots+m_n+r} +   n\sum_{m=0}^{n-1} c_m \sum_{m_1,\ldots,m_{n-1} \geq 0}
 c_{m_1}\cdots c_{m_{n-1}} a_{m_1+\ldots+m_{n-1}+m+r}. 
\end{align*}
The first term is finite because $c_m \leq c^{(n)}_m$ for $m \geq n$. The second term is finite by the induction hypothesis.
\end{proof}
\begin{lemma}
Suppose $X_1,X_2,X_3,\ldots$ is a sequence of smooth vector fields on $\R^d$ and  $\varphi \in C_c^\infty(\R^d)$.  Then, there exists a sequence $(c_m)_{ m \geq 0}$ of positive real numbers  such that, for any formal series $P(z) = \sum_{m \geq 0} a_m z^m \in \C[[z]]$ satisfying $|a_m|\leq c_m$ for all $m\geq0$, one has $\varphi \in \mathrm{dom}( P(X_{i_1}) \cdots P(X_{i_n}) )$ for all $n, i_1,\ldots,i_n \geq 1$.
\end{lemma}
\begin{proof}
Set
\[ M_m = \max \left\{ \left\| \tfrac{\partial^\alpha}{\partial x^\alpha} X_{i_1}^{m_1} \cdots X_{i_n}^{m_n} \varphi \right\| :  m_1+\ldots+m_n +|\alpha| \leq m \text{ and } n,i_1,\ldots,i_n   \leq m \right\}, \]
where $\|\cdot\|$ denotes the uniform norm. By the preceding lemma, there exists a sequence $(c_m)_{m \geq 0}$ of positive real numbers such that
\[ \sum_{m_1,\ldots,m_n \geq 0 } c_{m_1}\cdots c_{m_n} M_{m_1+\ldots+m_n+r} < \infty. \]
for all integers $n,r$ with $n \geq 1$ and $r \geq 0$. Then, provided $|a_m|\leq c_m$, one has that
\[ \sum_{m_1,\ldots,m_n \geq 0} |a_{m_1}| \ldots |a_{m_n}| \left\| \frac{\partial^\alpha}{\partial x^\alpha} X_{i_1}^{m_1} \cdots X_{i_n}^{m_n} \varphi \right\| < \infty \]
for all $n, i_1,\ldots,i_n \geq 1$ and  $\alpha \in \mathbb{N}^d$. Indeed, for all but finitely many terms of the above sum, we have $n,i_1,\ldots,i_n   \leq m_1+\ldots+m_n+ |\alpha|$, whence $\left\| \frac{\partial^\alpha}{\partial x^\alpha} X_{i_1}^{m_1} \cdots X_{i_n}^{m_n} \varphi \right\| \leq M_{m_1+\ldots+m_n+|\alpha|}$, by definition. It follows that $\varphi \in \dom(P(X_1) \cdots P(X_n))$. Indeed, 
\[ P(X_{i_1})\cdots P(X_{i_n})\varphi  = \sum_{m_1,\ldots,m_n \geq 0} a_{m_1}\cdots a_{m_n} X^{m_1}_{i_1} \cdots X^{m_n}_{i_n} \varphi, \]
where the above series is uniformly and absolutely convergent to a function $C_c^\infty(\R^d)$.
\end{proof}

As a direct corollary, we have:

\begin{lemma}\label{fast enough decay}
Suppose $M$ is a possibly non-Hausdorff smooth manifold and  $X_1,X_2,X_3,\ldots$ are smooth, non-branching vector fields on $M$. Fix a finite subset $\mathscr{P} \subseteq C_c^\infty(M)$. Then, there exists a sequence $(c_m)_{ m \geq 0}$ of positive real numbers  such that, for any formal series $P(z) = \sum_{m \geq 0} a_m z^m \in \C[[z]]$ satisfying $|a_m|\leq c_m$ for all $m\geq0$, one has $\varphi  \in \mathrm{dom}( P(X_{i_1}) \cdots P(X_{i_n}) )$ for all $\varphi \in \mathscr{P}$ and all $n, i_1,\ldots,i_n \geq 1$. \qed
\end{lemma}

\section{Lie groupoids and Lie groupoid actions}

In this somewhat lengthy section, we lay out the notations and conventions that will be used for Lie groupoids and their actions. For the most part, our conventions coincide with those of \cite{Mackenzie} and \cite{Renault[BOOK]}. We identify sections of the Lie algebroid with right-invariant vector fields, as is done in \cite{Mackenzie}. On the other hand, it will be convenient to have our vector fields and measures defined along the same fibers (namely the source fibers), so we will use right Haar systems instead of the left Haar systems used in \cite{Renault[BOOK]}. Although the material to follow is rather standard, we do sometimes include proofs, mainly in order to ensure that all arguments used make sense in the non-Hausdorff setting as well.

\subsection{Basic notions}

We will denote a typical Lie groupoid by $G$ and its unit space by $B$. The source and target maps are denoted $s,t : G \to B$. We always assume they are submersions. Multiplication is a smooth map from $G^{(2)} \coloneqq G\times_{s,t}G$ to $G$  that is performed from right to left and denoted by juxtaposition; given $\gamma_1,\gamma_2 \in G$, the product $\gamma_1\gamma_2$ is defined if and only if $s(\gamma_1)=t(\gamma_2)$. The inverse of $\gamma \in G$ is denoted by $\gamma^{-1}$. For convenience, it is always assumed that $B$ is embedded in $G$ as a closed submanifold. We write $k \coloneqq \dim(G)-\dim(B)$. In other words, $k$ is the dimension of the source and target fibers, for which we shall use the standard notations  $G_x \coloneqq s^{-1}(x)$ and $G^x \coloneqq t^{-1}(x)$, $x \in B$.

We allow the arrow space $G$, but not the unit space $B$, of a Lie groupoid to be non-Hausdorff. This is needed to accommodate certain examples of interest, such as groupoids arising from foliations. 
Dealing with the non-Hausdorff case  requires   modifications to the definitions one would use in the Hausdorff case. These kinds of issues are  well understood, see  \cite{Connes[1978]}, \cite{Connes[BOOK]}, \cite{Khoshkam-Skandalis[2004]}. We recall Hausdorffness of the unit space automatically implies that of the source and target fibers.

\begin{lemma}\label{hausfib}
Let  $G \rightrightarrows B$ be a possibly non-Hausdorff Lie groupoid. Then, for any $x \in B$, the source and target fibers $G_x$ and $G^x$ are Hausdorff.
\end{lemma}
\begin{proof}
See \cite{Tu}, Proposition~2.8.
\end{proof}

We also recall that the space of units always admits a Hausdorff open neighbourhood.

\begin{propn}\label{hdrffnbhd}
Let $G \rightrightarrows B$ be a Lie groupoid. Then, there exists a Hausdorff open set $W \subseteq G$ with $B \subseteq W$.
\end{propn}
\begin{proof}
This follows from \cite{Crainic-Fernandes}, Lemma~4.18. See also the discussion in Section~7 of \cite{Loizides-Sadegh-Sanchez}.
\end{proof}

\subsection{The Lie algebroid of a Lie groupoid}

\begin{defn}[Definition 3.1, \cite{Mackenzie}]
The  \textbf{Lie algebroid} of a Lie groupoid $G \rightrightarrows B$ is the vector bundle $AG \to B$ obtained by restricting the source fiber tangent bundle $\ker(ds) \subseteq TG$ to $B$.
\end{defn}

There is a canonical bundle map $AG \to TB$ called the \textbf{anchor map} given by restricting the differential of the target map $dt: TG \to TB$ to $AG$. Significantly, there is also a natural bracket operation on the sections of $AG$, but we will not need this.

\begin{defn}
A vector field $X$ on a Lie groupoid $G$ is said to be \textbf{right-invariant} if it is tangent to the source fibers of $G$ and, for all $\gamma \in G$, one has 
\begin{align*} 
(R_\gamma)_* X = X && \gamma \in G,
\end{align*}
where $R_\gamma$ denotes right-multiplication by $\gamma$. Note that $R_\gamma$ is a diffeomorphism  $G_{t(\gamma)} \to G_{s(\gamma)}$ and the above equation is implicitly  understood to refer to the restriction of $X$ to the relevant source fibers.
\end{defn}

It is easy to see that a right-invariant vector field  on $G$ is completely determined by its restriction to $B$, which is naturally a section of $AG$. Conversely, every section of $AG$ extends uniquely to a right-invariant vector field on $G$. We will freely denote a section of the Lie algebroid and its extension to  a right-invariant vector field by the same symbol. We note that the bracket of two right-invariant vector fields is easily seen to be right-invariant, leading to the bracket operation on sections of $AG$ (that we will not be needing). 

Because right-invariant vector fields are tangent to source fibers and source fibers are Hausdorff (Lemma~\ref{hausfib}), we have the following.

\begin{lemma}\label{compactimpliescomplete}
Every right-invariant vector field on a possibly non-Hausdorff Lie groupoid is non-branching, i.e. has a well-defined flow . \qed
\end{lemma}

It is common practice to denote the anchor map of a Lie algebroid by $\#$, but we will not do this here because it conflicts with the standard notation for fundamental vector field in the context of smooth actions, which we are also using. Instead, given a right-invariant vector field $X$ on $G \rightrightarrows B$, we write $X^B$ for the corresponding vector field on $B$ arising from the anchor map. To be more specific, the vector fields $X$ and $X^B$ are $t$-related.

In general, the flow of a right-invariant vector field on a Lie groupoid  need not be complete. A sufficient condition for the flow of a right-invariant vector field to be complete is that the  associated smooth section of the Lie algebroid is compactly-supported. See, for instance,  Proposition~3.6 in \cite{Francis[DM]}. If $X$ is a complete, right-invariant vector field on $G \rightrightarrows M$, then the induced vector field $X^B$ on the base is complete as well. Indeed, its flow is determined by:
\begin{align*}
e^{rX^B} b = t( e^{rX} b) && r \in \R, b \in B.
\end{align*}

\subsection{Smooth Haar systems}

We now turn to Haar systems. Their general theory appears in \cite{Renault[BOOK]} and their routine adaptation to the smooth case is discussed  in various sources such as   \cite{Brylinski-Nistor}.

\begin{defn}
Let $\pi : M \to N$ be a submersion of possibly non-Hausdorff manifolds. Assume that the fibers of $\pi$ are Hausdorff and $\lambda_y$ is a smooth measure on $\pi^{-1}(y)$ for each $y \in N$. We say that $\lambda=(\lambda_y)_{y \in N}$ is a \textbf{smooth system of measures} for $\pi$ if the following condition holds: if  $x \in M$,  $U \subseteq M$ is a neighbourhood of $x$, $V \subseteq N$ is a neighbourhood of $\pi(x)$, $\pi(U)=V$ and there are diffeomorphisms $\chi_U : U \to  \R^d \times \R^k$, $\chi_V : V \to \R^d$  (so  $\dim(M)=d+k$, $\dim(N)=d$) making the following diagram commutative
\[\begin{tikzcd}
U \ar[r,"{\chi_U}"] \ar[d,"\pi"] & \R^d \times \R^k \ar[d,"{\mathrm{proj}_2}"]  \\
V \ar[r,"{\chi_V}"] & \R^d, 
\end{tikzcd} \]
then the family of measures pushed forward through $\chi_U$ is the standard volume measure of $\R^k$ copied on vertical fibers, multiplied by a smooth, positive-valued function on $\R^d \times \R^k$.
\end{defn}

Working locally, one sees that  the above definition allows for integration along fibers of bump functions, even in the non-Hausdorff case where Definition~\ref{nonHausbump} is in effect. 

\begin{lemma}
Let $\pi : M \to N$ be a submersion of possibly non-Hausdorff manifolds with Hausdorff fibers that is  equipped with a  smooth system of measures $\lambda = (\lambda_y)_{y \in N}$. Then, integration  along fibers with respect to $\lambda$ defines a linear map $\pi_!: C_c^\infty(M) \to C_c^\infty(N)$.
\end{lemma}
\begin{proof}
Use Lemma~\ref{cover} to reduce to charts on which $\pi$ is a projection.
\end{proof}

\begin{defn}
Let $G \rightrightarrows B$ be a Lie groupoid. A \textbf{(smooth, right) Haar system} for $G$ is a smooth system of measures $\lambda = (\lambda_b)_{b \in B}$ for the source submersion $s : G \to B$ which is furthermore right-invariant in the sense that, for every $\gamma \in G$, the right multiplication map $R_\gamma$ is a measure preserving diffeomorphism $G_{t(\gamma)} \to G_{s(\gamma)}$.
\end{defn}

Fixing a Haar system $\lambda$ for a Lie groupoid $G$ turns $C_c^\infty(G)$ into a (generally noncommutative) algebra $C_c^\infty(G,\lambda)$ with respect to the convolution product $*$ defined by either of the following equivalent integrals:
\begin{align}\label{convdef}
(f*g)(\gamma_0) = \int_{G_{t(\gamma_0)}} f( \gamma^{-1}) g (\gamma \gamma_0) \ d\lambda_{t(\gamma_0)} = \int_{G_{s(\gamma_0)}} f( \gamma_0\gamma^{-1}) g (\gamma) \ d\lambda_{s(\gamma_0)}.
\end{align}
The isomorphism class of this algebra does not depend on the choice of Haar system. Indeed, if $\lambda$ and $\lambda'$ are two Haar systems for $G\rightrightarrows B$, then there is a unique smooth, positive-valued function $\rho$ on $B$ such that $\lambda' = (\rho \circ t) \lambda$. Moreover $\rho$  can be used to define a canonical algebra isomorphism $C_c^\infty(G,\lambda) \to C_c^\infty(G,\lambda')$. Specifically: 
\[ f \mapsto  \frac{1}{(\rho \circ s)^{1/2} (\rho \circ t)^{1/2}}f : C_c^\infty(G,\lambda) \to C_c^\infty(G,\lambda'). \]
Since the convolution algebra of $G$ is independent of $\lambda$, we will tend to write $C_c^\infty(G)$ instead of $C_c^\infty(G,\lambda)$. Note it is also possible to skirt the discussion of Haar systems completely in defining the convolution algebra of $G$ if one works with appropriate densities instead of functions, see \cite{Connes[BOOK]},~Section~2.5.

\subsection{Smooth groupoid actions}

We now turn our attention to groupoid actions.

\begin{defn}\label{actdef}
Let $G\rightrightarrows M$ be a Lie groupoid. Let $M$ be a possibly non-Hausdorff manifold together with a smooth map $\tau : M \to B$. Note the fiber product 
\[ G\ltimes M \coloneqq G\times_{s,\tau}M\]
is a closed submanifold of $G \times M$ because it is  the preimage of the diagonal under the submersion $s \times \tau : G \times M \to B \times B$. A \textbf{left action} of $G$ on $M$  is a smooth product $G\ltimes  M \ni (\gamma, x)  \mapsto \gamma \cdot x \in  M$, such that $\tau(\gamma \cdot x) = t(\gamma)$ for all $(\gamma, x) \in G \ltimes M$, $\tau(x) \cdot x = x$ for all $x \in M$ and $(\gamma_1\gamma_2)\cdot x = \gamma_1 \cdot (\gamma_2 \cdot x)$ for all $\gamma_1,\gamma_2 \in G$, $x \in M$ with $s(\gamma_1) = t(\gamma_2)$, $s(\gamma_2) = \mu(x)$. 

Similarly, if $M$ is equipped with a smooth map $\sigma : M \to G$, a \textbf{right action} of $G$ on $M$ is a smooth product $M \rtimes G  \ni (x,\gamma) \mapsto x \cdot \gamma \in M$, where $M\rtimes G\coloneqq M \times_{t,\sigma} G$,  such that $\sigma( x \cdot \gamma)=s(\gamma)$ for all $(x,\gamma) \in M \rtimes G$, $x \cdot \sigma(x) =x$ for all $x \in M$ and $x \cdot (\gamma_1\gamma_2) = (x\cdot \gamma_1)\cdot \gamma_2$ for all $\gamma_1,\gamma_2 \in G$, $x \in M$ with $\sigma(x) = t(\gamma_1)$, $s(\gamma_1)=t(\gamma_2)$.

Supposing $G_1\rightrightarrows B_1$ acts on $M$ from the left with respect to a map $\tau : M \to B_1$ and $G_2\rightrightarrows B_2$ acts on $M$ from the right with respect to the map $\sigma : M \to B_2$, we say that the left and right action \textbf{commute} with one another if $\gamma_1 \cdot (x \cdot \gamma_2) = (\gamma_1 \cdot x) \cdot \gamma_2$ for all $\gamma_1 \in G_1$, $\gamma_2 \in G_2$, $x \in M$ satisfying $s_1(\gamma_1)=\tau(x)$, $\sigma(x)=t_2(\gamma_2)$. In particular, $\sigma(\gamma_1 \cdot x) = \sigma(x)$, $\tau(x \cdot \gamma_2)=\tau(x)$.
\end{defn}

\begin{rmk}
Of course, we use the notation $G\ltimes M$ for $G\times_{s,\tau}M$ because it is the \emph{transformation groupoid}. There is a small snag here because we have insisted that our groupoids have Hausdorff unit spaces and $M$, which is the unit space of $G\ltimes M$, may be non-Hausdorff. In any event, we do not need to view $G\rtimes M$ as a groupoid in its own right here.
\end{rmk}

The following diagram summarizes some of the  data in Definition~\ref{actdef}
\[ \begin{tikzcd}
G_1 \ar[d,"{s_1}"'] \ar[r,bend left=50,yshift=-2ex] & X \ar[ld,"{\tau}"] \ar[rd,"{\sigma }"']& G_2 \ar[l,bend right=50, yshift=-2ex] \ar[d,"{t_2}"]\\
B_1 & & B_2 
\end{tikzcd} \]

One obvious example of commuting left and right actions  are the left and right multiplication action of a Lie groupoid on itself. Importantly, commuting left/right actions also appear in one formulation of the notion of Morita equivalence for Lie groupoids, see \cite{Ping}.

If $G \rightrightarrows B$ acts from the left on $\tau:M \to B$, then a right-invariant vector field $X$ on $G$ determines a corresponding vector field $X^M$ on $M$ (called the \textbf{fundamental vector field}) determined by 
\begin{align*}
X^M(m) = \tfrac{d}{dr}  (e^{rX} \tau(m)) \cdot m \Big|_{r=0} && m \in M.
\end{align*}
If $X$ is complete, then $X^M$ is complete and, indeed, the flow of $X^M$ is given by:
\begin{align*}
e^{rX^M}m = (e^{rX} \tau(m)) \cdot m && r \in \R, m \in M.
\end{align*}

If a Lie groupoid $G\rightrightarrows B$ with given right Haar system acts on $\tau : M\to B$ from the left, then $C_c^\infty(M)$ becomes a left $C_c^\infty(G)$-module with respect to the product 
\[C_c^\infty(G) \times C_c^\infty(M) \ni (f,\psi) \mapsto f * \psi \in C_c^\infty(M)\] defined by 
\begin{align}\label{leftmoddef}
(f*\psi)(x) = \int_{G_{\tau(x)}} f( \gamma^{-1}) g (\gamma \cdot x) \ d\lambda_{\tau(x)}.
\end{align}
To see this formula does indeed define a product (especially in the non-Hausdorff case where things may be  less clear), it is helpful to cast the formula \eqref{leftmoddef} in slightly more abstract form, as outlined below:

\begin{enumerate}
\item Let $G\rightrightarrows B$ be a Lie groupoid acting smoothly from the left on  $\tau:M\to B$. Here, $G$ and $M$ are possibly non-Hausdorff. Let $\pi ,\alpha : G \ltimes M \to M$ be the (restrictions of)   the second factor projection and the action map, respectively. Let $\jmath : G \ltimes M \to G \ltimes M$ be defined by $\jmath(\gamma,m)=(\gamma^{-1},\gamma \cdot m)$. Then, $\jmath$ is an order-2 diffeomorphism satisfying $\alpha=\pi\circ\jmath$. Indeed, although we do not make explicit use of it because its unit space $M$ may be non-Hausdorff, $\alpha, \pi$ and $\jmath$ are the target, source and inversion map for the transformation groupoid  $G \ltimes M$. The full set of structure maps are as follows:
\begin{align*}
\text{source: } &(\gamma,m) \mapsto m  \\
\text{target: } &(\gamma,m) \mapsto \gamma\cdot m \\
\text{product: } &(\gamma',\gamma\cdot m) (\gamma,m) = (\gamma'\gamma,m)\\
\text{inversion: } &(\gamma,m)   \mapsto (\gamma^{-1},\gamma\cdot m)
\end{align*}

\item Let $\lambda$ be the Haar system for $G$ and let $\widetilde \lambda$ be the obvious smooth system of measures for $\pi$ determined by copying $\lambda$ on fibers (one has  $\pi^{-1}(m)=G_{\tau(m)} \times \{m\}$ for $m \in M$). In other words, $\widetilde \lambda$ arises from the pullback square:
\[ \begin{tikzcd}
G \ltimes M \ar[r,"{\mathrm{pr}_1}"] \ar[d,"{\pi,\widetilde \lambda}"] & G \ar[d,"{s, \lambda}"] \\
M \ar[r,"\tau"] & B
\end{tikzcd} \]

\item Using the map $\jmath$, we also endow $\alpha : G\ltimes M\to M$ with a smooth system of measures (indeed, this is the left Haar system for the transformation groupoid) so that there is a fiberwise integration map $\alpha_! : C_c^\infty(G\ltimes M) \to C_c\infty(M)$ (Proposition~\ref{convsetup}~(3)). By construction, the following diagram is commutative:
\[ \begin{tikzcd}
C_c^\infty(G \ltimes M)   \ar[r,"{\jmath_*}"] \ar[d,"{\alpha_!}"] & C_c^\infty(G \ltimes M) \ar[d,"\pi_!"] \\
C_c^\infty(M) \ar[r,equals] & C_c^\infty(M) 
\end{tikzcd} \]

\item Given $f \in C_c^\infty(G)$ and $\varphi \in C_c^\infty(M)$, define $f\ltimes \varphi$ to be the restriction of $f \otimes \varphi$ to $G\ltimes M \subseteq G\times M$. 
Using  Proposition~\ref{convsetup} (1) and (2), one has that $f \ltimes \varphi \in C_c^\infty(G\ltimes M)$.

\item In terms of the above notations, the left $C_c^\infty(G)$-module structure \eqref{leftmoddef} of $C_c^\infty(M)$ may be expressed as follows:
\begin{align}\label{abstractconvsetup}
f * \psi \coloneqq \alpha_!( f \ltimes \psi ) && f \in C_c^\infty(G), \psi \in C_c^\infty(M).
\end{align}
\end{enumerate}

Similarly, if $G\rightrightarrows B$ acts on $\sigma:M \to B$ from the right, then $C_c^\infty(M)$ becomes a right $C_c^\infty(G)$-module with respect to the product $C_c^\infty(M) \times C_c^\infty(G) \ni (\psi,f) \mapsto \psi * f \in C_c^\infty(M)$ defined by
\begin{align}\label{rightmoddef}
(\psi * f)(x) = \int_{G_{\sigma(x)}} \psi( x \cdot \gamma^{-1}) g (\gamma) \ d\lambda_{\sigma(x)}.
\end{align}
Symmetrically  to the case of left actions, there is a canonical way to equip the right action map $\beta : M \rtimes G \to M$ with a smooth system of measures in terms of which the right $C_c^\infty(G)$-module structure may be alternatively expressed as follows:
\begin{align}\label{abstractrightmoddef}
\psi * f \coloneqq \beta_!( \psi \rtimes f ) && f \in C_c^\infty(G), \psi \in C_c^\infty(M). 
\end{align}
Again, although we do not make explicit use of it, we remark that the structure maps of the right transformation groupoid $M \rtimes G$  are given by:
\begin{align*}
\text{source: } &(m,\gamma) \mapsto m \cdot \gamma \\
\text{target: } &(m,\gamma) \mapsto m \\
\text{product: } &(m,\gamma) (m \cdot \gamma,\gamma') = (m,\gamma\gamma')\\
\text{inversion: } &(m,\gamma)   \mapsto (m\cdot\gamma,\gamma^{-1})
\end{align*}

If $M$ carries commuting actions of $G_1$ from the left and $G_2$ from the right, then $(f_1 * \psi)*f_2 = f_1 * (\psi * f_2)$ is satisfied for all $f_1 \in C_c^\infty(G_1)$, $f_2\in C_c^\infty(G_2)$, $\psi \in C_c^\infty(M)$ so that $C_c^\infty(M)$ has the structure of a $C_c^\infty(G_1)$-$C_c^\infty(G_2)$-bimodule. In the special case of a groupoid acting on itself from the left and right, this recovers the usual convolution product on $C_c^\infty(G)$.

In order to check the associativity of the bimodule structure, it is useful to pass through the two-sided transformation groupoid:
\[G_1 \ltimes M \rtimes G_2  
\coloneqq
\{ (\gamma_1,m,\gamma_2) \in G_1 \times M \times G_2 : s_1(\gamma_1)=\tau(m),\sigma(m)=t_2(\gamma_2) \} \]
whose structure maps are listed below:
\begin{align*}
\text{source: } &&(\gamma_1,m,\gamma_2) &\mapsto m \cdot \gamma_2   &\\
\text{target: } &&(\gamma_1,m,\gamma_2) &\mapsto \gamma_1 \cdot m &\\
\text{inversion: } &&(m,\gamma)   &\mapsto (m\cdot\gamma,\gamma^{-1}) &\\
\text{product: } &&(\gamma_1,m,\gamma_2) (\gamma_1',m',\gamma_2') 
&= (\gamma_1\gamma_1',(\gamma_1')^{-1}\cdot m,\gamma_2\gamma_2') & \text{(where $m \cdot \gamma_2 = \gamma_1'=m'$)}  \\
&&&= (\gamma_1\gamma_1',m'\cdot(\gamma_2)^{-1},\gamma_2\gamma_2').  &
\end{align*}
Because the actions commute, we have the following commuting diagram:
\[ \begin{tikzcd}
G_1 \ltimes M \rtimes G_2 \ar[rd,"{\mu}"] \ar[r,"{\alpha \times \id}"] \ar[d,"{\id \times \beta}"] & M \rtimes G_2 \ar[d,"{\beta}"] \\
G_1 \rtimes M \ar[r,"\alpha"] & M
\end{tikzcd} \]
where 
\[ \mu(\gamma_1,m,\gamma_2) \coloneqq \gamma_1 \cdot m \cdot \gamma_2. \]
Introducing fiberwise measures in the natural way, commutativity  holds as well at the level of bump functions and integration maps:
\[ \begin{tikzcd}
C_c^\infty(G_1 \ltimes M \rtimes G_2) \ar[rd,"{\mu_!}"] \ar[r,"{(\alpha \times \id)_!}"] \ar[d,"{(\id \times \beta)_!}"] & C_c^\infty(M \rtimes G_2) \ar[d,"{\beta_!}"] \\
C_c^\infty(G_1 \rtimes M) \ar[r,"\alpha_!"] & C_c^\infty(M)
\end{tikzcd} \]
leading to the desired associativity property:
\begin{align*}
f*(\psi*g) = (f*\psi)*g = \mu_!(f \ltimes \psi \rtimes g) && f\in C_c^\infty(G_1),\psi \in C_c^\infty(M), g \in C_c^\infty(G_2).
\end{align*}

\subsection{Dixmier-Malliavin for \texorpdfstring{$\R$}{R}-actions}

Suppose the additive group $\R$ acts smoothly on a possibly non-Hausdorff smooth manifold $M$. Let $X$ be the (non-branching, complete)  vector field on $M$ generating the action. The \emph{integrated form} of the action is the representation $\pi$ of $C_c^\infty(\R)$  on $C_c^\infty(M)$ defined by:
\begin{align}\label{Rint}
(\pi(f)\psi )(m) = \int_\R f(-t) \psi (e^{tX}m) \ dt. 
\end{align}
Of course, this is a very special case of the integrated forms of  groupoid actions just discussed. As in the more general case, especially when $M$ is non-Hausdorff, in order to see $\pi(f)$ does indeed define a mapping $C_c^\infty(M) \to C_c^\infty(M)$, it helps to express this action in the more abstract form:
\begin{align*} \pi(f) \varphi = \alpha_!( f \otimes \varphi) && f \in C_c^\infty(\R), \varphi \in C_c^\infty(M) \end{align*}
Here, $\alpha : \R \times M \to M$ is the action map, made into a measured submersion so as to render the following diagram commutative:
\[ \begin{tikzcd}
C_c^\infty(\R \times  M)   \ar[r,"{\jmath_*}"] \ar[d,"{\alpha_!}"] & C_c^\infty(\R \times  M) \ar[d,"(\mathrm{pr}_2)_!"] \\
C_c^\infty(M) \ar[r,equals] & C_c^\infty(M) 
\end{tikzcd},\]
where $\jmath(t,m)=(-t,e^{tX}m)$ and $\mathrm{pr}_2$ is  a measured submersion in the obvious way (copying Lebesgue measure).

As one expects, integration by parts enables  one to move differentiation  across the convolution map $C_c^\infty(\R) \times C_c^\infty(M) \to C_c^\infty(M)$. That is, $\pi(f')\varphi = \pi(f) X\varphi$ holds, where $X$ is the vector field generating the $\R$-action. We sketch a proof of this (quite routine) fact mainly to  indicate how it works in the formalism $\pi(f) \varphi = \alpha_!( f \otimes \varphi)$ which is convenient for rigorous treatment of the non-Hausdorff case.

\begin{lemma}
Let $\R$ act smoothly on a possibly non-Hausdorff smooth manifold $M$ via the complete vector field $X$. Let $\pi$ be the associated representation of $C_c^\infty(\R)$ on $C_c^\infty(M)$.  Then, $\pi(f')\varphi = \pi(f)X\varphi$ for all $f \in C_c^\infty(\R)$, $\varphi \in C_c^\infty(\R)$.
\end{lemma}
\begin{proof}
It is clear that $(\mathrm{pr}_2)_!:C_c^\infty(\R \times M) \to C_c^\infty(M)$ maps the image of the (non-branching) vector field $\langle \frac{d}{dt} , 0 \rangle$ on $\R \times M$ to zero. Conjugating by $\jmath(t,m)=(-t,e^{tX}m)$, this gives that integration along the action map $\alpha_!  :C_c^\infty(\R \times M) \to C_c^\infty(M)$ maps the image of $\langle -\frac{d}{dt},X \rangle$ to zero, giving the desired result.
\end{proof}
The above lemma leads to the following following preliminary factorization  result. This was  also the essential ingredient in \cite{Francis[DM]}.
\begin{thm}\label{prelimfac}
Let $\R$ act  smoothly on a smooth manifold $M$ via a complete vector field $X$ and let $\pi$ be the representation of $C_c^\infty(\R)$ on $C_c^\infty(M)$ defined by \eqref{Rint}. Suppose $f_0,f_1 \in C_c^\infty(\R)$ and $P(z) = \sum_{m \geq 0} a_m z^m \in \C[[z]]$ are such that $\delta = f_0+\sum_{m \geq 0}a_m f_1^{(m)}$, as in Lemma~\ref{DMlemma}.
Then, for any $\varphi \in \dom(P(X)) \subseteq C_c^\infty(M)$, one has:
\[
\varphi = \pi(f_0)\varphi+ \pi(f_1) P(X) \varphi.
\]
\end{thm}
\begin{proof}
See Theorem~8.1  in \cite{Francis[DM]}.
\end{proof}

It will be important for us to know that, if $M$ carries a (left) $\R$-action generated by a vector field $X$ that furthermore commutes with a given right action of a groupoid $G'$, then the operators $P(X) : \dom(P(X))\to C_c^\infty(M)$ for $P(z) \in \C[[z]]$ are right-linear for the $C_c^\infty(G')$-module structure of $C_c^\infty(M)$.  The case of interest for us is where $M$ carries commuting actions of groupoids $G$ and $G'$ from the left and the right, respectively, and the vector field on $M$ is the fundamental vector field $X^M$  associated to some section $X$ of the Lie algebroid $AG$. The desired statement is Corollary~\ref{rightlinearity} below, which follows directly from the three lemmas preceding it. The proofs of the these lemmas are straightforward, and we omit them.

\begin{lemma}\label{11}
Let $\pi : M \to N$ be a measured submersion of possibly non-Hausdorff smooth manifolds. Let $X$ be a non-branching smooth vector field  on $M$ and $Y$  a non-branching smooth vector field on $N$ and suppose $X$ and $Y$ are $\pi$-related. Let $P(z) \in \C[[z]]$. Then, $\pi_!$ maps $\dom(P(X))$ into $\dom(P(Y))$ and the following diagram is commutative:
\[ \begin{tikzcd}
\dom(P(X)) \ar[r,"{P(X)}"] \ar[d,"{\pi_!}"] \ar{d} & C_c^\infty(M)  \ar[d,"{\pi_!}"]\\
\dom(P(Y)) \ar[r,"{P(Y)}"] & C_c^\infty(N) 
\end{tikzcd} \]
\qed
\end{lemma}

Regarding the above, we note that, if  $X$ is non-branching, it is not automatic that $Y$ is non-branching,  as may be seen in the case where $\pi$ is the quotient map from two disjoint copies of $\R$ to the line with two origins. Of course, if  $Y$ is non-branching, it is also not automatic that $X$ is non-branching (one could take $N$ to be a single point).

\begin{lemma}\label{22}
Let $M$ be a  possibly non-Hausdorff smooth manifold and let $N$ be a closed submanifold of $M$. Let $X$ be a non-branching smooth vector field on $M$ that is tangent to $N$ and denote the restriction of $X$ to $N$ by $Y$.  Let $P(z)   \in \C[[z]]$. Then, the restriction map $C_c^\infty(M) \to C_c^\infty(N)$ maps  $\dom(P(X))$ into $\dom(P(Y))$ and the following diagram is commutative:
\[ \begin{tikzcd}
\dom(P(X)) \ar[r,"{P(X)}"] \ar[d,"\mathrm{restr}"] \ar{d} & C_c^\infty(M)  \ar[d,"\mathrm{restr}"]\\
\dom(P(Y)) \ar[r,"{P(Y)}"] & C_c^\infty(N) 
\end{tikzcd} \]
\qed
\end{lemma}

\begin{lemma}\label{33}
Let $M$ and $N$ be  possibly non-Hausdorff smooth manifolds. Let $X$ be a non-branching smooth vector field on $M$ and let $Y = \langle X,0\rangle$, a non-branching  smooth vector field on $M \times N$. Let $P(z)  \in \C[[z]]$.  Then, if  $f \in \dom(P(X)) \subseteq  C_c^\infty(M)$ and $g \in C_c^\infty(N)$, we have $f \otimes g \in \dom(P(Y))$ and $Y(f \otimes g) = (Xf) \otimes g$.
\qed
\end{lemma}

\begin{cor}\label{rightlinearity}
Let $G'\rightrightarrows B$ be a possibly non-Hausdorff Lie groupoid acting smoothly from the right on a possibly non-Hausdorff smooth manifold $M$ with respect to a smooth map $\sigma : M \to B$. Let $X$ be a smooth vector field on $M$ which commutes with the right action of $G'$ on $M$. Let $P(z)  \in \C[[z]]$. Then, $\dom(P(X))$ is right $C_c^\infty(G')$-submodule of $C_c^\infty(M)$, and $P(X)(\varphi * f )= (P(X)\varphi)  *f$ for all $\varphi \in \dom(P(X))$ and $f \in C_c^\infty(M)$.
\end{cor}
\begin{proof}
Let $\beta:M \rtimes G' \to M$  be the action map. Put  $Y \coloneqq \langle X,0\rangle$. Observe that $Y$ is a vector field on $M \times G'$  which is tangent to $M \rtimes G'$. Put $Z=Y|_{M \rtimes G'}$. Then we have
\begin{align*}
(P(X) \varphi) * f &=  \beta_!( (P(X) \varphi )\rtimes f) &&\\
&= \beta_!( (P(Y) (\varphi \otimes f))|_{M \rtimes G'}) && \text{by Lemma~\ref{22}} \\
&= \beta_!( P(Z) (\varphi \rtimes f)) && \text{by Lemma~\ref{11}} \\
&= P(X) \beta_!( \varphi \rtimes f)  && \text{by Lemma~\ref{33}} \\
&= P(X) ( \varphi * f), &&
\end{align*}
as required.
\end{proof}

\subsection{Ideals associated to invariant subsets}\label{idealsubection}

Next, we turn our attention to ideals in smooth groupoid algebras. A subset $Z$ of the unit space of a groupoid is called \textbf{invariant} if every arrow with source in $Z$ also has target in $Z$. Given a Lie groupoid $G\rightrightarrows B$ with Haar system, a closed, invariant set $Z \subseteq B$ and $p \in \N\cup\{\infty\}$,  one may consider the  $p$th order vanishing ideal $J_Z^p \subseteq C_c^\infty(G)$ consisting of the functions vanishing to (at least) $p$th order on $G_Z \coloneqq s^{-1}(Z)=t^{-1}(Z)$. We discuss these ideals in greater detail below.

\begin{defn}\label{vanishingorddef}
Let $M$ be a possibly non-Hausdorff smooth manifold, let $Z \subseteq M$ be a closed set, and let $p \in \mathbb{N}\cup\{\infty\}$.  We say that $\varphi \in C_c^\infty(M)$ \textbf{vanishes to $\textbf{p}$th order} on $Z$ if it can be decomposed as $\varphi = \sum_{i=1}^N (f_i \circ \chi_i)_0$  where  $U_i \subseteq M$ are chart neighbourhoods, $\chi_i : U_i \to \R^d$ are diffeomorphisms, and $f_i \in C_c^\infty(\R^d)$ are such that  $f_i$ and all its partial derivatives of order $<p$ vanish on $\chi_i(U_i \cap Z)$. The subscript $0$ above indicates the operation of extension by zero.  
\end{defn}

\begin{propn}
Let $G\rightrightarrows B$ be a Lie groupoid with given Haar system acting from the left on a smooth manifold $M$ with respect to a smooth map $\tau:M\to B$  so that $C_c^\infty(M)$ is a left $C_c^\infty(G)$-module. Let $Z \subseteq B$ be an invariant, closed subset and let $p,q \in \mathbb{N}\cup\{\infty\}$. If $f\in C_c^\infty(G)$ vanishes to $p$th order on $G_Z \coloneqq s^{-1}(Z) =t^{-1}(Z)$ and $\varphi \in C_c^\infty(M)$ vanishes to $q$th order on $M^Z \coloneqq \tau^{-1}(Z)$, then  $f * \varphi$ vanishes to $(p+q)$th order on $M^Z$.
\end{propn}
\begin{proof}
Let $\alpha : G \ltimes M \to M$ be the action map. It is simple to check in charts that $f \otimes \varphi \in C_c^\infty(G \times M)$ vanishes to order $p+q$ on $G_Z \times M^Z$. Restricting to $G \ltimes M  \subseteq G \times M$, one has that $f \ltimes \varphi$ vanishes to order $p+q$ on $(G_Z \times M^Z) \cap (G \ltimes M) =\alpha^{-1}(M^Z)$ (the latter equality uses the invariance of $Z$). The conclusion follows by expressing the convolution of $f$ and $\varphi$ in the form  $f * \varphi = \alpha_!(f \ltimes \varphi)$, as explained above.
\end{proof}

As a corollary of the above proposition, one has that the $p$th order vanishing ideals 
\[ J^p_Z \coloneqq \{ f \in C_c^\infty(G) :  \text{ $f$ vanishes to $p$th order on $G_Z$}\} \]
are indeed ideals for $Z\subseteq B$ closed and invariant and $p \in \mathbb{N} \cup \{\infty\}$. Moreover, $J^p_Z * J^q_Z \subseteq J^{p+q}_Z$. In \cite{Francis[DM]}, it was shown that this containment is an equality in the case where $Z \subseteq B$ is a closed, invariant submanifold. In particular, $J_Z^\infty*J_Z^\infty = J_Z^\infty$.  One of our goals in the present article is to show, more generally, that $J_Z^\infty$ is H-unital for $Z\subseteq B$ any closed, invariant subset.

We give one final definition/notation.  As a point of clarification, what follows are straightforward groupoid analogs of the fact that matrices of the form $\mat{* & * \\0 & 0}$, respectively $\mat{* & 0 \\ * & 0}$, constitute a right ideal, respectively left ideal, in the the algebra of $2$-by-$2$ matrices. Assume as above that $M$ carries commuting actions of $G_1$ from the left and $G_2$ from the right. Then, given a closed set $K \subseteq B_1$, we define
\begin{align*}
M^{K} &\coloneqq \tau^{-1}(K) \\
C_c^\infty(M)^{K} &\coloneqq   \{ \psi \in C_c^\infty(M) : \mathrm{supp}(\psi) \subseteq M^{K} \}. 
\end{align*}
It is straightforward to check that $C_c^\infty(M)^{K}$ is a right $C_c^\infty(G_2)$-submodule of $C_c^\infty(M)$. Symmetrically, if $K \subseteq B_2$ is a closed set, one may define
\begin{align*}
M_{K} &\coloneqq \sigma^{-1}(K) \\
C_c^\infty(M)_{K} &\coloneqq   \{ \psi \in C_c^\infty(M) : \mathrm{supp}(\psi) \subseteq M_{K} \}
\end{align*}
and  check that $C_c^\infty(M)_{K}$ is a left $C_c^\infty(G_1)$-submodule of $C_c^\infty(M)$.

\section{H-unitality of the convolution  algebra of a Lie groupoid}

In this section, we prove the first of our main results, the H-unitality of $C_c^\infty(G)$ for any Lie groupoid $G$ (Theorem~\ref{mainthm1} from the introduction). Indeed we prove a slightly stronger statement involving smooth actions of groupoids. The latter results are applicable, for instance, in  context of  Morita equivalences of Lie groupoids.

Throughout this section, $G\rightrightarrows B$ and $G'\rightrightarrows B'$ are Lie groupoids with given Haar systems and $M$ is a smooth manifold carrying commuting actions of $G$ from left and $G'$ from the right with respect to smooth maps $\tau:M \to B$ and $\sigma:M\to B'$. Thus, $C_c^\infty(M)$ has the structure of a $C_c^\infty(G)$-$C_c^\infty(G')$-bimodule. Recall as well the notations:
\begin{align*}
M^{K} &\coloneqq \tau^{-1}(K) \\
C_c^\infty(M)^{K} &\coloneqq   \{ \psi \in C_c^\infty(M) : \mathrm{supp}(\psi) \subseteq M^{K} \}. 
\end{align*}
 and the fact that $C_c^\infty(M)^{K}$ is a right $C_c^\infty(G')$-submodule of $C_c^\infty(M)$.

 \begin{lemma}
 Let $G\rightrightarrows B$ be a Lie groupoid with given Haar system. Let $X_1,\ldots,X_k \in C_c^\infty(B,AG)$, viewed as complete, right-invariant vector fields on $G$. Define $u:\R^k \times B \to G$ by 
\[ u(t_1,\ldots,t_k,b) = e^{t_1X_1}\cdots e^{t_kX_k}b \]
Suppose that $\widetilde W \subseteq \R^k \times B$ is an open set which is mapped diffeomorphically by $u$ onto an open set $W \subseteq G$. Then there is a linear bijection $\theta$ from $C_c^\infty(\widetilde W) \subseteq C_c^\infty(\R^k\times B)$ to $C_c^\infty(W) \subseteq C_c^\infty(G)$ given as pushforward by $u$, followed by multiplication by a suitable Jacobian factor with the property described below.

Let  $G$ act smoothly from the left on $\tau:M \to B$ and define 
\[ \widetilde\pi(f)\psi(m) = f(-t_k,\ldots,-t_1,\tau(e^{t_1X_1^M}\cdots e^{t_k X_k^M}m))\psi(e^{t_1X_1^M}\cdots e^{t_k X_k^M}m) .\]
Then, one has:
\[ \widetilde\pi(f) \psi = \theta(f) * \psi \]
for all $f \in C_c^\infty(\widetilde W)$ and all $\psi \in C_c^\infty(M)$.
 \end{lemma}

We will   generalize the following result from \cite{Francis[DM]}.

\begin{thm}[Theorem~5.1, \cite{Francis[DM]}]
Suppose $G \rightrightarrows B$ is a Lie groupoid with a given   Haar system and $M$ is a smooth manifold equipped with a left action of $G$. Thus, $C_c^\infty(M)$ has the structure of a $C_c^\infty(G)$-module. Then, for every $\varphi \in C_c^\infty(M)$, there exist $f_1,\ldots,f_N \in C_c^\infty(G)$ and $\psi_1, \ldots, \psi_N \in C_c^\infty(M)$ such that 
\[ \varphi = f_1 * \psi_1 + \ldots +  f_N * \psi_N. \] 
Moreover, this factorization can be taken such that, for all $i$, $\mathrm{supp}(\psi_i) \subseteq \mathrm{supp}(\varphi)$ and   $\mathrm{supp}(f_i) \subseteq W$, where $W$ is a prescribed open subset of $G$ containing $\tau( \mathrm{supp}(\varphi))$.
\end{thm}

As in the result above, we shall consider the more  general situation of a smooth groupoid action. This is done with an eye to certain applications to smooth Morita equivalences which may be explored elsewhere. The case of interest here in this article is that of a groupoid acting on itself.

\begin{thm}\label{actionunitality}
Let $\mathscr{P}$ be any finite subset of $C_c^\infty(M)$. Define  $K = \bigcup_{\varphi \in \mathscr{P}} \tau( \operatorname{supp}(\varphi))$ and let $W$ be any open subset of $G$ containing $K$. Then, there exist:
\begin{enumerate}
\item $f_1,\ldots,f_N \in C_c^\infty(W) \subseteq C_c^\infty(G)$.
\item a right $C_c^\infty(G')$-submodule $A_0$ with $\mathscr{P} \subseteq A_0 \subseteq C_c^\infty(M)^K$,
\item right $C_c^\infty(G')$-linear maps $\Psi_1,\ldots,\Psi_N : A_0 \to C_c^\infty(M)$ that do not increase supports
\end{enumerate}
such that, for all $\varphi \in A_0$, we have $\varphi = f_1 * \Psi_1(\varphi) + \ldots + f_N * \Psi_N(\varphi)$.
\end{thm}

We remark that, in view of Example~\ref{supbad}, (1) can in general be a stronger assertion than:  $f_i \in C_c^\infty(G)$ and $\supp f_i \subseteq W$. That being said, since the unit space of $G$ admits a Hausdorff neighbourhood (Proposition~\ref{hdrffnbhd}), this distinction is not important here as $W$ may without loss of generality be assumed to be Hausdorff.

\begin{proof}
Use the map $\tau : M \to B$ to endow $C_c^\infty(M)$ with the structure of a left $C^\infty(B)$-module as follows:
\begin{align*}
(\rho \cdot \varphi)(x) = \rho(\tau(x))\varphi(x) && \rho \in C^\infty(B), \varphi \in C_c^\infty(M), x \in M.
\end{align*}
It is simple to confirm that the left $C^\infty(B)$-module structure  on $C_c^\infty(M)$ commutes with the right $C_c^\infty(G')$-module structure.

Note that $K \subseteq B$ is compact by Proposition~\ref{pointset}. First we argue that it suffices to prove this theorem under the additional hypothesis that the Lie algebroid $AG$ is trivial (as a vector bundle) over some open neighbourhood of $K$. Indeed, assume this special case is already proven, and choose smooth, compactly-supported  functions $\rho_1,\ldots, \rho_N : B \to [0,1]$ such that $\sum_{i=1}^N \rho_i = 1$ holds on $K$ and  $AG$ is trivial on a neighbourhood of $\operatorname{supp}(\rho_i)$ for each $i$.  Define $\mathscr{P}^{(i)} = \{ \rho_i \cdot \varphi : \varphi \in \mathscr{P} \}$ for $i=1,\ldots,N$. By hypothesis, for $1 \leq i \leq N$, there exist
\begin{enumerate}
\item $f_1^{(i)},\ldots,f^{(i)}_{N_i} \in C_c^\infty(G)$ with supports contained in $W$, 
\item A right $C_c^\infty(G')$-submodule $A_0^{(i)} \subseteq C_c^\infty(M)$ with $\mathscr{P}^{(i)} \subseteq A_0^{(i)} \subseteq C_c^\infty(M)^K$,
\item Right $C_c^\infty(G')$-linear maps $\Psi_1^{(i)},\ldots,\Psi_{N_i}^{(i)} : A_0^{(i)} \to C_c^\infty(M)$ that do not increase supports
\end{enumerate}
such that, for all $\varphi \in A_0^{(i)}$, 
\[ \varphi = \sum_{j=1}^{N_i} f_j^{(i)} * \Psi_j(\varphi)  \]
Define:
\begin{align*}
A_0 &= \{ \varphi \in C_c^\infty(M)^K  : \rho_i \cdot \varphi  \in A_0^{(i)} \text{ for } i=1,\ldots,N \}.
\end{align*}
By definition, $\mathscr{P} \subseteq A_0 \subseteq C_c^\infty(M)^K$, and it is easy to check that $A_0$ is a right  $C_c^\infty(G')$-submodule. 

For $1 \leq i \leq N, 1 \leq j \leq N_i $, define  maps $\Psi_{ij} :A_0 \to C_c^\infty(M)$ by 
\[  \Psi_{ij}(\varphi) = \Psi_j^{(i)}(\rho_i \cdot \varphi). \]
It is easy to check these are linear maps which do not increase supports and are right $C_c^\infty(G')$-linear. Moreover, if $\varphi \in A_0$ 
\[ \varphi  = \sum_{i=1}^N \rho_i \cdot \varphi 
= \sum_{i=1}^N  \sum_{j=1}^{N_i} f^{(i)}_j * \Psi_j(\rho_i \cdot \varphi) 
= \sum_{i=1}^N \sum_{j=1}^{N_i} f_j^{(i)} * \Psi_{ij}(\varphi).
\]

It remains to check the case where $AG$ is trivial over a neighbourhood of $K$.  Then, by inverse function theorem,  there exists an open set $U \subseteq B$ with $K \subseteq U$ and  (complete by Lemma~\ref{compactimpliescomplete}) right-invariant vector fields $X_1,\ldots,X_k \in C_c^\infty(B,AG)$  such that 
\begin{align*}
u : \R^k \times B \to G && u(t_1,\ldots,t_k,b) = e^{t_1X_1} \cdots e^{t_k X_k} b 
\end{align*}
maps $\widetilde W \coloneqq (-1,1)^k \times U$ diffeomorphically onto an open subset of $G$ contained in $W$. Indeed by shrinking $W$ we can, and do, assume that $W=u(\widetilde W)$.

Let $X_1^M,\ldots, X_k^M$ denote the corresponding complete vector fields on $M$ and let $\pi_1,\ldots, \pi_k$ denote the corresponding representations of $C_c^\infty(\R)$ on $C_c^\infty(M)$ given by \eqref{Rint}. Because the action of $G$ on $M$ commutes with the action of $G'$ on $M$, the $\R$ actions determined by the fundamental vector fields $X_1^M,\ldots,X_k^M$ commute with the action of $G'$ as well.

By Lemma~4.2 in \cite{Francis[DM]}, there is a linear bijection $\theta_W$ from $C_c^\infty(\widetilde W) \subseteq C_c^\infty(\R^k\times B)$ to $C_c^\infty(W) \subseteq C_c^\infty(G)$ given as pushforward by $u$, followed by multiplication by a suitable smooth Jacobian factor such that
\[ \widetilde\pi(f) \psi = \theta_W(f) * \psi \]
for all $f \in C_c^\infty(\widetilde W)$ and all $\psi \in C_c^\infty(M)$, where
\[ \widetilde\pi(f)\psi(m) \coloneqq f(-t_k,\ldots,-t_1,\tau(e^{t_1X_1^M}\cdots e^{t_k X_k^M}m))\psi(e^{t_1X_1^M}\cdots e^{t_k X_k^M}m) .\]
In particular, fixing $\rho \in C_c^\infty(B)$ such that $\rho=1$ on $K$ (so that $\rho\cdot \varphi = \varphi$ for all $\varphi \in C_c^\infty(M)^K$) and $\supp(\rho)\subseteq U$, we get that
\begin{align}\label{chartconv}
 \pi_1(f_1) \cdots \pi_k(f_k) \psi =  \theta_W( f_k \otimes \ldots \otimes f_1 \otimes \rho) * \psi \end{align}
for all $f_1, \ldots,f_k \in C_c^\infty(-1,1)\subseteq C_c^\infty(\R)$ and all $\psi \in C_c^\infty(M)^K$ (see also Lemmas 4.1 and 4.3 in \cite{Francis[DM]}).

By Lemma~\ref{fast enough decay}, there exists a sequence of positive reals $(c_m)_{m \geq 0}$ such that, for any formal series $P(z) = \sum_{m \geq 0 } a_m z^m$ with $|a_m|\leq c_m$, we have
\[ \mathscr{P} \subseteq  \mathrm{dom}( P(X_{i_1}^M) \cdots P(X_{i_n}^M) ) \] 
for all $n, i_1,\ldots,i_n \geq 1$. Using Lemma~\ref{DMlemma}, write $\delta_0 = f_0 +\sum_{m \geq 0} f_1^{(m)}$ where $f_0,f_1 \in C_c^\infty(-1,1) \subseteq C_c^\infty(\R)$   and $|a_m| \leq c_m$. By repeated application of Theorem~\ref{prelimfac}, we obtain
\begin{align*}
 \varphi 
=& \pi_1(f_0) \varphi + \pi_1(f_1) P(X^M_1) \varphi \\
=& \pi_1(f_0) \pi_2(f_0) \varphi  + \pi_1(f_0) \pi_2(f_1) P(X^M_2) \varphi  \\
&+ \pi_1(f_1) \pi_2(f_0) P(X_M^1) \varphi + \pi_1(f_1)\pi_2(f_1)P(X^M_2)P(X^M_1) \varphi \\
\vdots& \\
=& \sum_{i_1, \ldots,i_k \in \{0,1\}} \pi_1(f_{i_1}) \cdots \pi_k(f_{i_k}) P(X^M_k)^{i_k} \cdots P(X^M_1)^{i_1} \varphi 
\end{align*}
for all $\varphi \in \mathscr{P}$, with the understanding that $P(X_i)^0=\id$. Thus, using \eqref{chartconv},
\[ \varphi = \sum_{i_1, \ldots,i_k \in \{0,1\}} f_{i_1,\ldots,i_k} *  \Psi_{i_1,\ldots,i_k} ( \varphi ) \]
where 
\begin{align*}
f_{i_1,\ldots,i_k} &\coloneqq \theta_W( f_1 \otimes \cdots \otimes f_k \otimes \rho) \subseteq C_c^\infty(W) \subseteq C_c^\infty(G) \\ 
A_0 &\coloneqq  \bigcap_{i_1, \ldots,i_k \in \{0,1\}} \dom\left( P(X^M_k)^{i_k} \cdots P(X^M_1)^{i_1} \right) \cap C_c^\infty(M)^K \\
\Psi_{i_1,\ldots,i_k} &\coloneqq P(X^M_k)^{i_k} \cdots P(X^M_1)^{i_1}|_{A_0}
\end{align*}
By Corollary~\ref{rightlinearity}, we have that $A_0$ is a right $C_c^\infty(G')$-submodule of  $C_c^\infty(M)^K$  and that the maps $\Psi_{i_1,\ldots,i_k}$ are right $C_c^\infty(G')$-linear, as required. They do not increase supports by Proposition~\ref{nosuppinc}.
\end{proof}

Specializing to the case where $G$ acts on itself from the right and left, Theorem~\ref{actionunitality} amounts to the following.

\begin{cor}\label{Gfactor}
Let $G \rightrightarrows B$ be a Lie groupoid with a given Haar system. Let $\mathscr{P}$ be a finite subset of $C_c^\infty(G)$ and put $K \coloneqq \bigcup_{\varphi\in\mathscr{P}}t(\supp \varphi)$ (by Proposition~\ref{pointset}, $K$ is a compact subset of $B$) and let $W\subseteq G$ be an open set with $K \subseteq W$ (by Lemma~\ref{hdrffnbhd}, we may assume $W$ is Hausdorff). Then, there exist:
\begin{enumerate}
\item $f_1,\ldots,f_N \in C_c^\infty(W) \subseteq C_c^\infty(G)$,
\item a right ideal $A_0 \subseteq C_c^\infty(G)$ with $\mathscr{P} \subseteq A_0 \subseteq C_c^\infty(G)^K$,
\item right $C_c^\infty(G)$-linear maps $\Psi_1,\ldots,\Psi_N:A_0 \to C_c^\infty(G)$ that do not increase supports
\end{enumerate}
such that, for all $\varphi \in A_0$, we have $\varphi = f_1*\Psi_1(\varphi)+\ldots+f_N*\Psi_N(\varphi)$. \qed
\end{cor}

The desired result (Theorem~\ref{mainthm1} from the introduction) now follows.

\begin{cor}
For any Lie groupoid $G$ with given Haar system, the smooth convolution algebra $C_c^\infty(G)$ is H-unital.
\end{cor}

\begin{proof}
This follows from Corollary~\ref{Gfactor} and Proposition~\ref{suffcond}, taking $\phi$ to be  the map 
\[ \varphi \mapsto f_1 \otimes \Psi_1(\varphi) + \ldots + f_N \otimes \Psi_N(\varphi). \]
\end{proof}

\section{Factorization of flat functions relative to a submersion}

In this section, we extend the following known factorization result for functions that are ``flat'' on a given closed set so as to allow factorization to be performed relative to a submersion.

\begin{thm}
Suppose $M$ is a Hausdorff smooth manifold and $Z \subseteq M$ is closed. Let $I_Z^\infty$ denote the ideal in $C^\infty_c(M)$ consisting of functions that vanish to infinite order (i.e. vanish with all derivatives)  on $Z$. Then, given $\varphi_1,\ldots,\varphi_N \in I_Z^\infty$, there exists $\rho, \psi_1,\ldots,\psi_N \in I_Z^\infty$ such that $\varphi_i=\rho \psi_i$ for $i=1,\ldots,N$. 
\end{thm}
\begin{proof}
This follows from Theorem~3.2 of \cite{Voigt}. See also the proof of Theorem~6.2 in \cite{Wodzicki} (``property [F]'' is defined on pp. 612).
\end{proof}

Our extended factorization result is Theorem~\ref{subfac} below, essentially a ``with parameters'' version of the preceding result. It seems likely that Theorem~\ref{subfac} is known to experts but, as a reference could not be located, we provide a proof.

 Let us first record the following elementary extension principle.

\begin{lemma}\label{elemext}
Let $Z \subseteq \R^d$ be a closed set and let $f \in C^\infty(W)$,  $W \coloneqq \R^d \setminus Z$. If $f$ and all its partial derivatives vanish at the boundary of $W$, then $f$ extends by zero to a smooth function on $\R^d$ vanishing to infinite order  on $Z$. 
\end{lemma}
\begin{proof}
Let $g$ be the extension by zero of $f$. 
Using induction and the standard result that a function on $\R^d$ is $C^1$ if and only each of its first order partials exists and is continuous, one only needs to check that the first order partials of $g$ exists and vanish on $Z$. One may therefore reduce to the case $d=1$ where the desired conclusion may be deduced from the mean value theorem. 
\end{proof}

 The preceding lemma leads to directly to a condition under which one may form the quotient of two smooth functions vanishing to infinite order on a given closed set.

\begin{lemma}\label{elemquot}
Let $Z \subseteq \R^d$ be a closed set. Suppose  $f , g \in C^\infty(\R^d)$ vanish to infinite order on $Z$  and $g >0$  on $W\coloneqq \R^n \setminus Z$. If, for every $\alpha \in \N^d$ and $m \in \N$, the function $\frac{\partial^\alpha f}{g^m} \in C^\infty(W)$  vanishes at the boundary of $W$, then the extension by zero of $\frac{f}{g}$ is a smooth function on $\R^d$ vanishing to infinite order on $Z$. 
\end{lemma}

\begin{proof}
Let $\mathscr{F} \subseteq C^\infty(W)$ denote the  $C^\infty(\R^d)$-linear span of the functions $\frac{\partial^\alpha f}{g^m}$ for $\alpha \in \N^d$, $m \in \N$. By assumption, the functions in $\mathscr{F}$ all vanish at the boundary of $W$. Observe that $\mathscr{F}$ is closed under taking partial derivatives. Indeed, if $\alpha \in \N^d$ is a multi-index, $m \in \N$, $h \in C^\infty(\R^d)$ and $\partial$ is one of the first order partials, then 
\[ \partial ( h \frac{\partial^\alpha f}{g^m} )= (\partial h) \frac{\partial^\alpha f}{g^m} + h \frac{(\partial \circ \partial^\alpha) f}{g^m} - m(\partial g)h \frac{\partial^\alpha f}{g^{m+1}}.\]
Thus, thinking of $\frac{f}{g}$ as a smooth function on $W$,  we have by induction that all of its higher order partial derivatives vanish at the boundary of $W$. Thus,  $\frac{f}{g}$ extends to a smooth function on $\R^d$ that vanishes to infinite order on $Z$ by  Lemma~\ref{elemext}.
 \end{proof}

We use  the existence of a regularized distance function for an arbitrary  closed subset of Euclidean space. Such distance functions appear in the proof of the Whitney extension theorem given in \cite{Stein[1970]} and can be explicitly constructed using cubical meshes. Note that Wodzicki \cite{Wodzicki} also uses these functions by way of an  appeal to \cite{Voigt}.

\begin{thm}\label{regdist}
Let $Z \subseteq \R^d$ be closed,  $W \coloneqq \R^d \setminus Z$ and let  $\delta(x)$ denote the distance from $x$ to $Z$. Then, there is a smooth function $\Delta : W \to (0,\infty)$ such that: 
\begin{enumerate}[(i)]
\item there exist constants $c_2>c_1>0$ such that $c_1 \delta(x)  \leq \Delta(x)  \leq c_2 \delta(x)$, $x \in W$, 
\item for each $\alpha \in \N^d$, there is  a constant $B_\alpha \geq 0$ such that $| \partial^\alpha \Delta(x) | \leq B_\alpha \delta(x)^{1-|\alpha|}$, $x \in W$. 
\end{enumerate}
The constants $c_1,c_2, B_\alpha$ are independent of $Z$. 
\end{thm}
\begin{proof}
This is Theorem~2 on pp.~171 of \cite{Stein[1970]}.
\end{proof}

Next, we apply Theorem~\ref{regdist} to construct smooth defining functions $\rho \in C^\infty(\R^d)$ for arbitrary closed sets $Z \subseteq \R^d$ with rate of vanishing  governed  by an arbitrary smooth function $g : (0,\infty) \to (0,\infty)$ vanishing together with all  derivatives at $0$.

\begin{lemma}\label{decayatZ}
Let $Z \subseteq \R^d$ be closed and put $W \coloneqq \R^d \setminus Z$. Let $\Delta:W \to (0,\infty)$ be a regularized distance function for $Z$, as in Theorem~\ref{regdist}. Then, given any $g \in C^\infty(0,\infty)$ vanishing together with all derivatives at $0$,  the extension by zero of  $g\circ \Delta$ is a smooth function $\rho$  on $\R^d$ vanishing to infinite order on $Z$. 
\end{lemma}
\begin{proof}
In view of Lemma~\ref{elemext}, we just need to show that all of the partial derivatives of $g\circ \Delta$ vanish at the boundary of $W$. For nonzero $\alpha \in \N^d$, we  have the following formula for the $\alpha$th partial derivative  of a composition:
\[ \partial^\alpha(g \circ \Delta) = \sum_{k=1}^{|\alpha|} ( g^{(k)} \circ \Delta) \sum_{\beta \in J(\alpha,k)} C(\alpha,k,\beta) \prod_{j=1}^k \partial^{\beta(j)} \Delta \]
where $C(\alpha,k,\beta) \in \N$ and  
\[ J(\alpha,k) \coloneqq\{ \beta = (\beta(1),\ldots, \beta(k)) \in (\N^d)^k : \beta(j) \neq 0,  j=1,\ldots,k \text{ and } {\textstyle\sum_{j=1}^k} \beta(j) = \alpha \}. \]
See the proof of Lemma~3.1 in \cite{Voigt} as well as \cite{Fraenkel}. Applying the estimates of Theorem~\ref{regdist} to the above expression for $\partial^\alpha(g\circ\Delta)$, it is straightforward to derive an estimate of the form
\[ |\partial^\alpha (g \circ \Delta)| \leq \sum_{k=1}^{|\alpha|} C(k)  (g^{(k)} \circ \Delta) \Delta^{k-|\alpha|}, \]
where $C(k) \geq 0$. Since $g(t)t^{-p} \to 0$ as $t\to 0^+$ for all $p \in \N$ and $\Delta$ vanishes at the boundary of $W$, the estimate above shows that $\partial^\alpha(g\ \circ \Delta)$ vanishes at the boundary of $W$ as needed.
\end{proof}

  Lemma~\ref{decayatZ}, together with the following result, allows one to construct smooth defining functions for a closed set $Z \subseteq \R^d$ which vanish to infinite order on $Z$ ``as slowly as desired''. 
\begin{lemma}\label{slowdecay}
Let $f_k$ be a  sequence of functions on $(0,\infty)$ such that  $\lim_{t \to 0^+} f_k(t) t^{-p} = 0$ for all $p \in \N$. Then, there exists a smooth, positive-valued function $g$ on $(0,\infty)$  that vanishes together with all its derivatives at $0$ such that  $\lim_{t \to 0^+} \frac{f_k(t)}{g(t)} =0$ for all $k$.
\end{lemma}
\begin{proof}
See Lemma~6.7 in \cite{Francis[DM]} for a complete proof and references to the literature.
\end{proof}

\begin{thm}\label{localfac}
Let $Z \subseteq \R^k$ be closed, $W \coloneqq \R^k \setminus Z$. Suppose $f \in C^\infty(\R^k \times \R^\ell)$ vanishes to infinite order on $Z \times \R^\ell$. Then, there exists $\rho \in C^\infty(\R^k)$ vanishing to infinite order on $Z$ and strictly positive on $W$ and $g \in C^\infty(\R^k \times \R^\ell)$ vanishing to infinite order on $Z \times \R^k$ such that $f(x,y) = \rho(x) g(x,y)$ for all $(x,y) \in \R^k\times \R^\ell$. 
\end{thm}

\begin{proof}
Because $f$ vanishes with all its derivatives on $Z \times \R^\ell$, it follows (e.g. from Taylor's theorem) that  $(x,y)\mapsto f(x,y)\delta(x)^{-p}$ vanishes at the boundary of $W\times\R^\ell$  for any $p \in \N$, where $\delta(x)$ denotes the distance from $x$ to $Z$. The same is true if $\delta$ is replaced by a regularized distance function $\Delta : W \to (0,\infty)$ (Theorem~\ref{regdist}). Given $\alpha \in \N^k \times \N^\ell$, $m \in \N$, $r>0$, define $f_{\alpha,m,r} : (0,\infty)\to(0,\infty)$ by
\[ f_{\alpha,m,r}(t) = \sup \{ |\partial^\alpha f(x,y)|^{1/m}  : |x|,|y| \leq r \text{ and }
\Delta(x) \leq t \}. \]
By design, $f_{\alpha,m,r}$ is an increasing, continuous function satisfying 
\[ \lim_{t\to0^+} f_{\alpha,m,r}(t) t^{-p} = 0 \]
for all $p \in \N$. Thus, by Lemma~\ref{slowdecay}, there exists a smooth function $g : (0,\infty)\to(0,\infty)$ that vanishes with all its derivatives at $0$ such that 
\[ \lim_{t\to0^+} \frac{f_{\alpha,m,r}(t)}{g(t)} =0 \]
for all $\alpha \in \N^k\times\N^\ell$, $m\in\N$, $r>0$. By Lemma~\ref{decayatZ}, $g \circ \Delta$ extends by zero to a smooth function $\rho : \R^d \to [0,\infty)$ which vanishes to infinite order on $Z$. The estimate 
\[ \left| \frac{ \partial^\alpha f (x,y)}{\rho(x)^m} \right| \leq
\left( \frac{f_{\alpha, m,r}(\Delta(x))}{g(\Delta(x))} \right)^m \]
(valid for $|x|,|y| \leq r$) shows that $(x,y) \mapsto \frac{\partial^\alpha f(x,y)}{\rho(x)^m}$  vanishes at the boundary of $W \times \R^\ell$ for all $\alpha \in \N^k\times\N^\ell$, $m\in\N$ and so, by Theorem~\ref{elemquot}, $(x,y) \mapsto \frac{f(x,y)}{\rho(x)}$ extends by zero to a smooth function on $\R^k\times\R^\ell$ vanishing to infinite order on $Z \times \R^k$. 
\end{proof}

We are now in a position to state and prove the main result of this section.

\begin{thm}\label{subfac}
Let $\pi : M \to B$ be a submersion where $B$ is Hausdorff and $M$ is possibly non-Hausdorff. View $C_c^\infty(M)$ as a $C^\infty(B)$-module with module structure given by $f \cdot \varphi \coloneqq (f \circ \pi)\varphi$ for $f \in C^\infty(B)$, $\varphi \in C_c^\infty(M)$. Let $Z \subseteq B$ be a closed set and let $\varphi_1,\ldots, \varphi_N \in C_c^\infty(M)$ vanish to infinite order on $\pi^{-1}(Z)$ (Definition~\ref{vanishingorddef}). Then, there exists $\rho \in C^\infty(B)$ vanishing to infinite order on $Z$ and strictly positive on $W\coloneqq B \setminus Z$ and $\psi_1,\ldots,\psi_N \in C_c^\infty(M)$ vanishing to infinite order on $\pi^{-1}(Z)$ such that $\varphi_i=\rho \cdot \psi_i$ for $i=1,\ldots,N$. 
\end{thm}

\begin{proof}
Suppose  $f,\rho_1,\rho_2$  are smooth functions on $\R^k \times \R^\ell$ vanishing to infinite order on $\{0\} \times \R^\ell$ and that $0 < \rho_1 < \rho_2$  on the complement of $\{0\} \times \R^\ell$. We remark that, if $f/\rho_1$ extends to a smooth function on $\R^k \times \R^\ell$ vanishing to infinite order on $\{0\} \times \R^\ell$, then $f/\rho_2$ also extends to a smooth function on $\R^k \times \R^\ell$ vanishing to infinite order on $\{0\} \times \R^\ell$ (see Lemma 6.9~\cite{Francis[DM]}). This remark, together with the fact that $B$ is Hausdorff and therefore  admits smooth partitions of unity allow one to (i) consider only the local problem, where $M$ and $N$ are Euclidean spaces and $\pi$ is projection, and (ii) consider the case of only a single function $\varphi \in C_c^\infty(M)$. The local case of a single function is given by Theorem~\ref{localfac}.
\end{proof}

\section{H-unitality of ideals arising from  invariant subsets}

In this final section, we prove the second of our main results, the H-unitality of infinite order vanishing ideals in smooth groupoid algebras (Theorem~\ref{mainthm2} from the introduction).  This leads directly to an excision principle for invariant, closed subsets.  Permanence properties of H-unitality (Theorem~\ref{permanence}) give analogous results for Whitney functions as well. Applications of this excision result will be considered elsewhere.

Recall from Wodzicki's seminal paper on H-unitality,  if $Z$ is a closed subset of a Hausdorff smooth manifold $M$, then the ideal in $C^\infty(M)$ consisting of functions that vanish together with all derivatives on $Z$ is H-unital (Theorem~6.1, \cite{Wodzicki}). The goal here is to obtain the noncommutative analogue.

The main ingredient of the proof is the following direct corollary of Theorem~\ref{subfac} (take the submersion $\pi$ to be the target submersion).

\begin{cor}\label{targetfac}
If $G \rightrightarrows B$ is  a Lie groupoid, $Z \subseteq B$ is an invariant, closed subset and   $\varphi_1,\ldots,\varphi_N$ belong to the infinite order vanishing ideal $J_Z^\infty \subseteq C_c^\infty(G)$ (see Section~\ref{idealsubection}), then there exists a smooth function $\rho \in C^\infty(B)$,  vanishing to infinite order on $Z$ and positive on $B\setminus Z$ and $\psi_1,\ldots,\psi_N \in J_Z^\infty$ such that $\varphi_i=\rho \cdot \psi_i$ for $i=1,\ldots,N$. \qed
\end{cor}

Note that, if $\rho \in C^\infty(B)$ is nonvanishing on $B\setminus Z$, then  $\varphi \mapsto \rho \cdot \varphi : J^\infty_Z \to J^\infty_Z$ is clearly injective. Indeed, if $G$ is Hausdorff and $Z$ is a closed submanifold, then $\varphi \mapsto \rho \cdot \varphi$ is injective on all of $C_c^\infty(G)$. When $G$ is non-Hausdorff and $Z\subseteq B$ is a closed submanifold (or more generally has empty interior), then injectivity of $\varphi \mapsto \rho\cdot\varphi$ on all of $C_c^\infty(G)$ may fail, as the following example shows.
\begin{ex}
Let $B= \R$ and let $G = \R^ \times \sqcup  \Z$, the ``line with infinitely many origins'' with its obvious non-Hausdorff smooth manifold structure. Then, $G$ is a Lie groupoid over $B$ where  $s=t$ is the obvious projection to $G \to B$ and multiplication $G^{(2)} = \R^\times \sqcup \Z^2 \to G$ is $(\id_{\R^\times} \sqcup \mathrm{addition})$.  Define $\varphi = 0 \sqcup f$ where $f : \Z \to \{0,1\}$ is given by $f(1)=1$, $f(-1)=-1$, $f(n)=0$ for $n \neq \pm 1$. It is easy to see that $\varphi$ is a (nonzero) element of $C_c^\infty(G)$. However, $\rho \cdot \varphi =0$ for any $\rho \in C^\infty(\R)$ vanishing at $0$.
\end{ex}

As in Section~5, we deduce H-unitality from a technical result designed to be used with Proposition~\ref{suffcond}.

\begin{thm}
Let $G \rightrightarrows B$ be a Lie groupoid with a given Haar system  and let $Z \subseteq B$ be a  $G$-invariant, closed subset. Let $\mathscr{P}$ be a finite subset of $J_Z^\infty$. Put $K \coloneqq \bigcup_{\varphi\in\mathscr{P}}t(\supp \varphi)$  and let $W\subseteq G$ be an open set with $K \subseteq W$. Then, there exist:
\begin{enumerate}
\item $f_1,\ldots,f_N \in J_Z^\infty \subseteq C_c^\infty(G)$,
\item a right ideal $A_0 \subseteq J_Z^\infty$ with $\mathscr{P} \subseteq A_0 \subseteq C_c^\infty(G)^K$,
\item right $C_c^\infty(G)$-linear maps $\Psi_1,\ldots,\Psi_N:A_0 \to J_Z^\infty$ that do not increase supports
\end{enumerate}
such that, for all $\varphi \in A_0$, we have $\varphi = f_1*\Psi_1(\varphi)+\ldots+f_N*\Psi(\varphi)$.
\end{thm}
\begin{proof}
Let $\mathscr{P} = \{ \varphi_1,\ldots,\varphi_n\}$. From Corollary~\ref{targetfac} above, there exists $\rho \in C^\infty(B)$,  vanishing to infinite order on $Z$ and positive on $B\setminus Z$,  and $\psi_1,\ldots,\psi_n \in J^\infty_Z$ such that $\varphi_i = \rho \cdot \psi_i$ for $i=1,\ldots,n$. We have that $\varphi \mapsto \rho \cdot \varphi$ is a right $C_c^\infty(G)$-linear bijection of $J_Z^\infty$ onto the right ideal $\rho \cdot J_Z^\infty \subseteq J_Z^\infty$. Let $M_{\frac{1}{\rho}}: \rho\cdot J_Z^\infty \to J_Z^\infty$ denote the inverse isomorphism of right $C_c^\infty(G)$-modules. By design, $M_{\frac{1}{\rho}}(\psi_i)=\varphi_i$ for $i=1,\ldots,n$.

From Corollary~\ref{Gfactor}, there exist:
\begin{enumerate}
\item $g_1,\ldots,g_N \in C_c^\infty(W) \subseteq C_c^\infty(G)$,
\item a right ideal $B_0 \subseteq C_c^\infty(G)$ with $\mathscr{P} \subseteq B_0 \subseteq C_c^\infty(G)^K$,
\item right $C_c^\infty(G)$-linear maps $\Phi_1,\ldots,\Phi_N:B_0  \to C_c^\infty(G)$ that do not increase supports
\end{enumerate}
such that, for all $\varphi \in B_0$, we have $\varphi = g_1*\Phi_1(\varphi)+\ldots+g_N*\Phi(\varphi)$. Define:
\begin{align*}
f_i = g_i \cdot \rho && A_0 \coloneqq B_0 \cap \rho \cdot J_Z^\infty && \Psi_i \coloneqq M_{\frac{1}{\rho}} \circ \Phi_i|_{A_0}  && i=1,\ldots,N
\end{align*}
so that
\begin{align*}
\varphi &= g_1*\Phi_1(\varphi)+\ldots+g_N*\Phi(\varphi) \\
&= g_i*(\rho\cdot \Psi_1(\varphi))+\ldots+g_N*(\rho\cdot \Psi_N(\varphi)) \\
&= f_1*\Psi_1(\varphi)+\ldots+f_N*\Psi(\varphi) 
\end{align*}
and the $f_i$, $A_0$ and $\Psi_i$ are as needed.
\end{proof}

As a corollary, we obtain the desired H-unitality result and its consequence for excision (Theorem~\ref{mainthm2} and Corollary~\ref{excision1} from the introduction).

\begin{cor}
For any Lie groupoid $G\rightrightarrows B$ with given Haar system and any $G$-invariant, closed subset $Z \subseteq B$, the associated ideal $J_Z^\infty \subseteq C_c^\infty(G)$ is H-unital. Consequently, the short exact sequence
\[ 
\begin{tikzcd}
0 \ar[r] & J_Z^\infty \ar[r] & C_c^\infty(G) \ar[r] & C_c^\infty(G)/J_Z^\infty \ar[r] & 0 
\end{tikzcd}
\]
induces a corresponding long exact sequence in cyclic/Hochschild homology.
\end{cor}
\begin{proof}
Follows from Proposition~\ref{suffcond} and Theorem~\ref{wod}.
\end{proof}

Finally, we briefly reiterate the corollaries  for Whitney functions which were already discussed in the introduction.

\begin{defn}
Let $G\rightrightarrows B$ be a Lie groupoid with given Haar system and $Z \subseteq B$ a closed,  invariant submanifold with $J_Z^\infty \subseteq C_c^\infty(G)$ the corresponding ideal.  The convolution algebra of compactly-supported \textbf{Whitney functions} on $G_Z$ is defined as the quotient
\[ \mathcal{E}^\infty_c(G_Z) \coloneqq C_c^\infty(G)/J_Z^\infty \] 
(this notation disguises the fact that $\mathcal{E}^\infty_c(G_Z)$ depends on the inclusion of $G_Z$ in $G$, rather than only on $G_Z$). 
\end{defn}

\begin{rmk}
Taking $G$ Hausdorff, this  is not the classical definition of a Whitney function, but,  by Whitney's extension theorem \cite{Whitney}, it is equivalent. 
\end{rmk}

\begin{cor}
For any Lie groupoid $G\rightrightarrows B$ with given Haar system and any $G$-invariant, closed subset $Z \subseteq B$, the algebra of noncommutative Whitney functions $\mathcal{E}_c^\infty(G_Z)$ is H-unital.
\end{cor} 
\begin{proof}
Since $C_c^\infty(G)$ and $J_Z^\infty$ are H-unital, this follows from permanence of H-unitality under  quotients (Theorem~\ref{permanence}).
\end{proof}

\begin{cor}
Let $G\rightrightarrows B$ be a Lie groupoid with given Haar system. Let $Z \subseteq Y \subseteq B$ be $G$-invariant, closed subsets of  $B$, then $J^\infty_{Z,Y} \coloneqq \ker\left(\mathcal{E}_c^\infty(G_Y) \to \mathcal{E}_c^\infty(G_Z)\right)$ is H-unital. Consequently, the short exact sequence
\[ \begin{tikzcd} 
0 \ar[r] & J_{Z,Y}^\infty \ar[r] & \mathcal{E}_c^\infty(G_Z) \ar[r] &  \mathcal{E}_c^\infty(G_Y) \ar[r] & 0
\end{tikzcd}
\]
induces corresponding long exact sequences in cyclic and Hochschild homology.
\end{cor}
\begin{proof}
Noting $J_{Z,Y}^\infty \cong {J^\infty_{Z}}/{J^\infty_{Y}}$, the H-unitality of $J_{Z_1,Z_2}^\infty$ follows from Theorem~\ref{permanence}. The statement concerning long exact sequences follows from Theorem~\ref{wod}.
\end{proof}





\bibliographystyle{abbrv}
\bibliography{Francis}

\end{document}